\documentclass[a4paper]{amsart}

\usepackage{amsmath,amsthm,amssymb, color}

\usepackage{amsmath,amssymb,amsbsy,amsfonts,latexsym,amsthm,amscd,amsxtra,
amssymb}
\newtheorem{Theo}{Theorem}%[section]
\newtheorem{Def}{Definition}
\newtheorem{Prop}{Proposition}
\newtheorem{Cor}{Corollary}
\newtheorem{Lem}{Lemma}
\newtheorem{Rema}{Remark}
\newenvironment{Rem}{\begin{Rema} \begin{upshape}} {\end{upshape}\end{Rema}}
\newtheorem*{Pf}{Proof}
\newenvironment{Proof}{\begin{Pf} \begin{upshape}} {\end{upshape} \qed\end{Pf}}

\newcommand\beqa[1]{ \begin{eqnarray} \label{#1}}
\newcommand{\eeqa}{ \end{eqnarray} }
\newcommand{\beqano}{ \begin{eqnarray*} }
\newcommand{\eeqano}{ \end{eqnarray*} }

\newcommand{\T}{ {\mathbb T}   }

\newcommand{\R}{ {\mathbb R}   }
\newcommand{\Z}{ {\mathbb Z}   }

\renewcommand{\H}{ {\mathbb H}   }

\renewcommand \a {\alpha}

\newcommand \tx {\tilde{x}}

\newcommand \e {\varepsilon }

\renewcommand \d {\delta}
\newcommand \D{\Delta}
\newcommand \m {\mu}

\newcommand \g {\gamma}

\newcommand \G {\Gamma}
\newcommand \s {\sigma}
\renewcommand \S {\Sigma}

\renewcommand \th {\theta}
\renewcommand \l {\lambda}

\newcommand \tz {{\tilde{z}}}
\renewcommand \th {{\tilde{h}}}
\newcommand \tM {{\widetilde{M}}}
\newcommand \tH {{\widetilde{H}}}
\newcommand \uY {{\underline{Y}}}
\newcommand \uZ {{\underline{Z}}}
\newcommand \tPsi {{\widetilde{\Psi}}}

\newcommand \cA {{\mathcal A}}

\newcommand \cY {{\mathcal Y}}

\newcommand \cH {{\mathcal H}}

\newcommand \cS {{\mathcal S}}

\newcommand{\pa}{{p_{\alpha}}}
\newcommand{\pb}{{p_{\beta}}}
\newcommand{\pc}{{p_{\gamma}}}

\newcommand \calM {{\mathfrak M}}
\newcommand \calC {{\mathfrak C}}

\def\ie{\hbox{\it i.e.\ }}

\newcommand \rH {{\rm H}}
\newcommand \rT {{\rm T}}
\newcommand \supp {{\rm supp\ }}

\newcommand \dist {{\rm dist\ }}

\begin{document}

\title[Symplectic and contact properties of the Ma\~n\'e critical value]{Symplectic and contact properties of the Ma\~n\'e critical value of the universal cover}

\author{Gabriel P. Paternain and Alfonso Sorrentino}
\address{Department of Mathematics and Mathematical Statistics, University of Cambridge, Wilberforce Road, Cambridge CB3 0WB, United Kingdorm.}
\email{g.p.paternain@dpmms.cam.ac.uk}
\address{Dipartimento di Matematica e Fisica, Sezione di Matematica, Universit\`a degli Studi Roma Tre, Largo San Leonardo Murialdo1, 00146 Rome, Italy.}
\email{sorrentino@mat.uniroma3.it}
\date{\today}

%%%%%%%%%%%%%%%%%%%%%%%%%%%%%%%%%%%%%%%%%%%%%%%%
%%%%%%%%%%%%%%%%%% ABSTRACT %%%%%%%%%%%%%%%%%%%%
\begin{abstract} We discuss several symplectic aspects related to the Ma\~n\'e critical value $c_u$ of the universal cover of a Tonelli Hamiltonian. In particular we show that the critical energy level is never of virtual contact type  for manifolds of dimension greater than or equal to three.
We also show the symplectic invariance of the finiteness of the Peierls barrier and the Aubry set of the universal cover.  We also provide an example where $c_u$ coincides with the infimum of Mather's $\alpha$ function but the Aubry set of the universal cover is empty and the Peierls barrier is finite.
A second example exhibits all the ergodic invariant minimizing measures with zero homotopy, showing, quite surprinsingly, that the union of their supports is not a graph, in contrast with Mather's celebrated graph theorem.

\end{abstract}

%\vspace{5pt}
%%%%%%%%%%%%%%%%%%%%%%%%%%%%%%%%%%%%%%%%%%%%%%%%%%%%%%%%%%%%%%%%%%%%%%%%%%%%%%%%%%%%%%%%%%%%%%%%%%%%%%

\maketitle

\section{Introduction} 

Let $M$ be a compact connected manifold of dimension $n$ and let us consider  an autonomous Hamiltonian $H: \rT^*M \longrightarrow \R$, which is $C^2$, strictly convex ({\it i.e.}, with positive definite hessian) and superlinear in each fiber,  and the corresponding Lagrangian system $L: \rT M \longrightarrow \R$, which is defined by Fenchel-Legendre duality. Hamiltonians and Lagrangians of this kind are often said to be of ``{\it Tonelli type}''.  

Since the seminal works by John Mather \cite{Mat91, Mat93}, Ricardo Ma\~n\'e \cite{Mane91} and Albert Fathi \cite{Fathibook}, much effort has been spent in order to study the dynamics of these systems and their symplectic properties, both using variational methods -- the so--called {\it principle of least action} -- and their analytical counterpart, in the form of viscosity solutions and subsolutions of Hamilton-Jacobi equations. See for instance, just to mention a few references, \cite{CIPP98,  CI99, C01, CP02, FathiSiconolfi, FathiMaderna, MassartSorrentino, PP97, SorTAMS, ButlerSorrentino}.

The energy values on which these methods can be applied are called {\it Ma\~n\'e  critical values}. These values appear in several different contexts and can be defined and interpreted in many interesting ways, each reflecting and encoding a distinct dynamical or symplectic significance (see for instance \cite{B07, CIPP98, PPS, SorViterbo}). 

This work aims at advancing further the work in \cite{CFP}, where the authors thoroughly analyzed the relation between these critical values and the symplectic topology of the corresponding energy hypersurfaces. More specifically, they focused on understanding how the dynamical, symplectic and contact properties of the regular energy levels of the Hamiltonian change when one passes through some Ma\~n\'e  critical value. 
Here we will focus more on the Peierls barrier and the Aubry set.

\subsection{Ma\~n\'e critical values}\label{sec1.2}

In the case of a compact $M$, a crucial idea behind the definition of these critical values is the following observation.  If one modifies the Lagrangian (and consequently the corresponding Hamiltonian) by subtracting a closed $1$-form, then it is easy to verify that while this does not affect the Euler-Lagrange flow of the system, nevertheless it has a substantial impact on its {\it action-minimizing} properties.  More specifically, if  $\eta$ denotes a closed $1$-form on $M$, and $L_{\eta}(x,v) := L(x,v)- \langle \eta(x),\,v\,\rangle$  and $H_{\eta}(x,p)=H(x,\eta(x)+p)$ are respectively the modified Lagrangian and Hamiltonian, then the corresponding {\it Ma\~n\'e  critical value} $c(H_{\eta})$ can be defined  in many equivalent ways (see also \cite{SorLectNotes}):

\begin{itemize}
\item[1)] {\sc Variational definition I}: if $\calM_{L}$ denotes the set of invariant probability measures of $L$ (hence of $L_{\eta}$), then:
\beqano
c(H_{\eta}) := - \min_{\m \in \calM_L} \int_{\rT M} L_{\eta}(x,v)\, d\m\,,
\eeqano
{\it i.e.}, it is the opposite of the {\it minimal average $L_{\eta}$-action} of invariant probability measures of the Euler-Lagrange flow of L  (see \cite{Mat91}). A measure which realizes this minimum is called an {\it action--minimizing} (or {\it Mather's}) measure of cohomology class $[\eta]$.\\

\item[2)] {\sc Variational definition II}: for any absolutely continuous curve (abs. cont.) $\g:[a,b]\longrightarrow M$, we define its {\it $L$-action} as
$$A_{L}(\g):= \int_a^b L(\g(t),\dot{\g}(t))\,dt.$$ 
Then (see \cite{Mane91, CI99}):
\beqano
c(H_{\eta})&:=&\inf  \{k\in\R:\; A_{L_{\eta}+k}(\g)\geq 0, \quad \forall\;\mbox{abs. cont. loop}\; \g\} \nonumber\\
&=&\sup\{k\in\R:\;  A_{L_{\eta}+k}(\g)< 0 \;  \mbox{for some abs. cont. loop}\; \g \}\,. \\
\eeqano

\item[3)] {\sc Hamiltonian definition}: in \cite{Carneiro} Dias Carneiro proved that $c(H_{\eta})$ represents the energy of  action-minimizing measures of cohomology class $[\eta]$, {\it i.e. } the energy level on which they are supported.\\

\item[4)] {\sc Symplectic definition}: it was proved in \cite{PPS} that  $c(H_{\eta})$ represents the infimum of the energy values $k$'s such that the energy sublevel $\{H(x,p)\leq k\}$ contains a smooth Lagrangian graph of cohomology class $[\eta]\in H^1(M;\R)$. In particular, it corresponds to the smallest energy sublevel containing Lipschitz Lagrangian graphs of cohomology class $[\eta]$.\\

\item[5)] {\sc PDE definition:} $c(H_{\eta})$ is the unique $k$ for which Hamilton-Jacobi equation $H(x,\eta(x)+p)=k$ admits viscosity solutions \cite{Fathibook}.\\

\end{itemize}

The above definitions (and many others) are all equivalent (in the compact case) and it turns out -- as it could be easily evinced for instance from item (4) -- that they only depend on the cohomology class of $\eta$ and not on the particular representative that has been chosen.
Moreover, these values are somehow symplectic invariant, in the sense that they are  invariant under the action of exact symplectomorphisms \cite{B07}; non-exact symplectomorphisms  do preserve the set of values $\{c(H_{\eta})\}_{[\eta]\in H^1(M;\R)}$, but they do affect the corresponding cohomology classes: they essentially act as a ``translation'' in the parameter ({\it i.e.}, the cohomology class).\\

It is rather useful to consider all of these values as a function on $\rH^1(M;\R)$:
\beqano
\a: H^1(M;\R) &\longrightarrow & \R \\
c & \longmapsto & c(H_{\eta_c}),
\eeqano
where $\eta_c$ represents any closed $1$-form of cohomology class $c$. 
This function, which is usually called {\it Mather's $\a$-function}, turns out to  be convex and superlinear (see \cite{Mat91}) and it surprisingly behaves as a sort of  ``effective Hamiltonian'' for the system; moreover, its regularity and strict convexity properties encode many interesting aspects of the dynamics of the system \cite{Massart, MassartSorrentino, SorViterbo}. \\

The minimum of this function, which is usually denoted  by $c_0(H)$, is called {\it Mane's strict critical value}.
This value is a symplectic invariant (not only for exact symplectomorphisms) \cite{PPS}: it corresponds to the largest energy sublevel that does not contain in its interior any Lagrangian submanifold Hamiltonianly isotopic to the zero section. In particular, observe that $c_0(H)$ represents the lowest energy level in which  these variational methods (known as {\it Aubry--Mather theory}) can be applied.  \\

In some cases, there is a way to push these methods below the strict critical value. The main idea consists in lifting the system to a {\it cover space} (see \cite{CIPP98}).

Given $\Pi: \widehat{M} \longrightarrow M$ a cover of $M$, we can consider the corresponding lifted Hamiltonian $\widehat{H}:= H\circ d\Pi$  and the associated Lagrangian $\widehat{L}$.
Following the variational definition $(2)$ from the compact case, one can define the { Ma\~n\'e  critical value} associated to this cover, as
$$ c(H,\widehat{M}) : =  c(\widehat{H})= \inf  \{k\in\R:\; A_{\widehat{L}+k}(\g)\geq 0, \quad \forall\;\mbox{abs. cont. loop}\; \g\}$$
(observe that  {\it a-priori} it is not clear whether all previous characterizations that are valid in the compact case, do still hold for non-compact cover spaces).

It is not difficult  to verify that $c(H,\widehat{M}) \leq c({H})$. Moreover,  equality holds if $\Pi: \widehat{M} \longrightarrow M$ is a finite cover (see \cite[Lemma 2.2]{CP02}).\\

Two distinguished covers are the {\it universal cover} $\Pi_u: \widetilde{M} \longrightarrow M$ and the {\it abelian cover} $\Pi_a: \overline{M} \longrightarrow M$ (\ie the cover of $M$ whose  group of deck transformations is $H_1(M;\Z)$). We shall denote the respective critical values by 
\beqa{criticalvalues}
c_u(H):=c(H,\widetilde{M}) \qquad {\rm and} \qquad c_a(H):=c(H,\overline{M}).
\eeqa
Clearly, $c_u(H)\leq c_a(H) \leq c_0(H)$.  Moreover, in \cite{PP97}  it was proved that in the case of compact manifolds,  $c_a(H)=c_0(H)$. Therefore, we can conclude that  also in the case of the abelian cover, $c_a(H)$ is also a symplectic invariant.\\

\begin{Rem}
Recall that a discrete group $G$ is said to be {\it amenable}, if there is a left (or right) invariant mean on $\ell^{\infty}(G)$, the space of all bounded functions on $G$.
For example, all finite groups or abelian groups are amenable; similarly for finite extensions of solvable groups. On the other hand, if a group contains a free subgroup of two generators then it is not amenable; this is the case of the fundamental group of a compact surface of genus $g \geq  2$.  In \cite{PP97}, the authors provided an example of a Tonelli Hamiltonian $H$ on a compact surface of genus two, for which $c_u(H)<c_0(H)$. See also Section \ref{examples}.
\end{Rem}

\subsection{Main results}

Recall that a hypersurface $\Sigma$ in a symplectic manifold $(V^{2n},\omega)$ is of {\it contact type} if $\omega\big|_{\Sigma} = d\lambda$ for a contact form $\lambda$ on $\Sigma$, {\it i.e.}, a $1$-form such that $\lambda\wedge (d\lambda)^{n-1}$ is nowhere vanishing. \\

A related notion is the notion of {\it virtual contact structure}. A hypersurface $\Sigma$ is said to be of virtual contact type if 
$\pi^*\omega\big|_{\widetilde{\Sigma}} = d\lambda$, for a contact form $\lambda$ on the universal cover $\pi: \widetilde{\Sigma}\longrightarrow \Sigma$ such that
$$
\sup_{x\in \widetilde{\Sigma}} |\lambda_x| \leq C < +\infty \quad {\rm and}\quad
\inf_{x\in \widetilde{\Sigma}} \lambda( R) \geq \epsilon >0,
$$
where $|\cdot|$ is a metric on $\Sigma$ and $R$ is a vector field generating ${\rm Ker}(\omega|\Sigma)$  (both pulled back to $\widetilde{\Sigma}$).
If $(\Sigma,\omega)$ is virtually contact and its fundamental group $\pi_1(\Sigma)$ is amenable, then $(\Sigma,\omega)$ is of contact type 
(this follows from a standard argument using amenability as in \cite{FathiMaderna,P06}).

For energy levels $\Sigma_k:=\{H(x,p)=k\}$ of Tonelli Hamiltonians we have the following (assume $M$ compact and, for simplicity, orientable):

\begin{enumerate}
\item for $k>c_0(H)$, the energy level $\Sigma_k$ is of contact type \cite[Theorem B.1]{Contreras};
\item for $M$ different from the 2-torus and $c_u(H)<k\leq c_0(H)$, the energy level $\Sigma_{k}$ is never of contact type \cite[Theorem B.1]{Contreras};
\item for $k>c_u(H)$ the energy level $\Sigma_k$ is virtually contact \cite{CFP} (this was proved for magnetic Lagrangians with potentials, but it is easy to see that it holds for any Tonelli Hamiltonian);
\item in \cite[Section 5]{CMP} the authors provide an example of a Tonelli Lagrangian on $T\T^2$ for which
$c_u(L)=c_a(L)$ and the corresponding energy level is of contact type.
\end{enumerate}
Observe, in particular, that if $\Sigma_k$ is of contact (or virtual contact) type, then $k$ is a regular value of the energy function
$E(x,v)=\frac{\partial L}{\partial v} (x,v)\cdot v - L(x,v)$.
In particular, $k> e:= \max_{x \in M} E(x,0) =- \min_{x \in M} L(x,0)$, which corresponds by superlinearity to
$$
e = \min\left\{c\in\R:\; \pi: E^{-1}( c) \subset TM \longrightarrow M \; \mbox{is surjective}\right\}.\\
$$ 

\vspace{10 pt}

Continuing this analysis, we prove the following  result (see section \ref{sec2}).\\

\noindent {\bf Theorem A.} {\it If $\dim M\geq 3$ and $H: \rT^*M\longrightarrow \R$ is Tonelli, then the energy level $\S_{c_u(H)}:=\{E(x,v)= c_u(H)\}$ is never of virtual contact type,
where $E(x,v)=\frac{\partial L}{\partial v} (x,v)\cdot v - L(x,v)$ denotes the energy function}.

\bigskip

Observe that the hypothesis $\dim M \geq 3$ is necessary because of the example in item (4) above. However it makes sense to ask:

\medskip

\noindent{\bf Question I}: Is it possible to find examples of Tonelli Hamiltonians on surfaces of higher genus, for which  $\S_{c_u(H)}$ is of virtual contact type?\\

In Sections \ref{sec3}, \ref{sec4} and \ref{sec5} we investigate the symplectic properties of the  action-minimizing objects associated to the lift of the Hamiltonian to the universal cover.  Our main result in this direction is a proof of the symplectic invariance of $c_u(H)$, the finiteness of the Peierls barrier $h_{\widetilde{H}}$ and the Aubry set $\cA^*_{\widetilde{H}}$ (we review the definition of these objects in Section \ref{sec3}).

\bigskip

\noindent {\bf Theorem B.} {\it Let $M$ be a closed manifold and $H: \rT^*M\longrightarrow \R$ a Tonelli Hamiltonian.
Assume $\Psi:T^*M\to T^*M$ is a symplectic diffeomorphism such that $H':=H\circ\Psi$ is still of Tonelli type.
Then
\begin{enumerate}
\item $c_{u}(H)=c_{u}(H')$;
\item The Peierls barrier $h_{\widetilde{H}}$ is finite if and only if $h_{\widetilde{H'}}$ is finite;
\item The projected Aubry set $\cA_{\widetilde{H}}$ is empty if and only if  the projected Aubry set $\cA_{\widetilde{H'}}$ is empty;
\item $\cA^*_{\widetilde{H}'}=\widetilde{\Psi}^{-1}(\cA^*_{\widetilde{H}})$, where $\widetilde{\Psi}$ is any lift of $\Psi$ to $T^*\widetilde{M}$.
\end{enumerate}

}

\medskip

The main difficulty in this setting is represented by the lack of compactness, which might create quite peculiar aftermaths, like the Peierls barrier being infinite or  the Aubry set (or Peierls set) being empty. Examples of these occurrences can be found in \cite[Section 6]{C01}, however
these examples are unfortunately {\it not} lifts of Lagrangians on closed manifolds. We remedy this here by providing in Section \ref{examples} 
two examples. One has $c_u=c_0$, finite Peierls barrier but empty Aubry set in the universal cover and the other has
$c_u<c_0$, and also finite Peierls barrier and empty Aubry set. In the latter we find all the minimizing ergodic invariant measures with zero homotopy thus illustrating what we are up against.

\medskip

\noindent{\bf Question II}: Is it possible to find examples of Tonelli Hamiltonians on closed manifolds for which $h_{\widetilde{H}}$ is infinite?\\

The proof of the first three items in Theorem B is not difficult (see Proposition \ref{thm1}, Theorems \ref{thm2} and  \ref{thm3}). The proof of item (4) is more involved and for this we need to
adapt Bernard's methods in \cite{B07} to this non-compact setting (see Theorem \ref{invarianceAubry}).

%%%%%%%%%%%%%%%%%%%%%%%%%%%%%%%%
%%%%%%%%%%%%%%%%%%%%%%%%%%%%%%%%
%%%%%%%%%%%%%%%%%%%%%%%%%%%%%%%%

%%%%%%%%%%%%%%%%%%%%%%%%%%%%%%%%%%

%%%%%%%%%%%%%%%%%%%%%%%%%%%%%%%%%%%%%%%%%%%%%%
% Section: Virtually contact

\section{Virtual contact property}\label{sec2}

In this section we want to prove Theorem A. Let us start by recalling some definitions and properties.
Define the space of continuous functions with at most linear growth 
$$
C^0_{\ell}(\rT M):=\left\{f \in C^0(\rT M;\R):\; 
\sup_{(x,v)\in \rT M} \frac{|f(x,v)|}{1+\|v\|}<\infty \right\}
$$
and consider the set
$$
\calM_{\ell} := \left\{\m\; \mbox{Borel probability measures on}\; \rT M \; \mbox{s.t.}\; \int_{\rT M}\|v\|\,d\m<\infty    \right\}
$$
endowed with the topology: $\lim_n \m_n =\m$ if and only if $\lim_{n} \int f\,d\m_n = \int f\,d\m$ for any $f\in C^0_{\ell}(\rT M)$.
Observe that $\calM_{\ell}$ can be naturally embedded into the dual space $\left(C^0_{\ell}(\rT M) \right)^*$ and its topology coincides with the weak$^*$-topology on  $\left(C^0_{\ell}(\rT M) \right)^*$. It is also possible to prove that this topology is metrizable.

For any $\g:[0,T]\longrightarrow M$ absolutely continuous curve, let us associate a Borel probability measure $\m_{\g}$ uniformly distributed on $\g$, {\it i.e.},
$$
\int_{\rT M} f\,d\m_{\g} = \frac{1}{T} \int_0^T f(\g(t),\dot{\g}(t))\,dt \qquad \forall\;f\in C^0_{\ell}(\rT M).
$$
Since $\int_0^T \|\dot{\g}(t)\|\,dt < \infty$, then $\m_{\g}\in \calM_{\ell}$. We denote by 
$\calC(M)$ the set of $\m_{\g}$ generated by closed absolutely continuous loops $\g$ and consider its closure $\overline{\calC(M)}$ in $\calM_{\ell}$. A measure in $\overline{\calC(M)}$ is called a {\em holonomic measure}. It is easy to check that this set is convex and that it contains all invariant probability measures for any Tonelli Lagrangian $L$ (it is essentially Birkhoff's theorem). See for instance \cite{Mane91,CI99}.\\

Amongst holonomic measures, we want to look at the special ones generated by contractible loops. More specifically, let 
$$
\calC_0(M) := \left\{\m_{\g}\in \calC(M):\; \g\;\mbox{is a contractible abs. cont. loop}  \right\}
$$
and let us consider its closure
$\cH_0(M):= \overline{\calC_0(M)} \subset \overline{\calC(M)} \subset \calM_{\ell}$. We call these measures ``{\em holonomic measures with zero homotopy type}''.\\

\begin{Rem}
Recalling the definition of $c_u(H)$ and observing that for any given $\m\in \cH_0(M)$ there exists a sequence $\m_n\in \calC_0(M)$ such that   $\m_n\rightarrow \m$ and 
$\int L\, d\m_n \longrightarrow \int L\, d\m$, then:
$$
c_u(H):= -\inf_{\m\in \cH_0(M)} A_L(\m)\,.
$$
\end{Rem}
\vspace{10 pt}

In \cite{Carneiro} Dias Carneiro proved that Mather's minimizing measures for $L$ have support
contained in the energy level $c(H)$. Our first proposition establishes a weaker result
for the action--minizimizing measure with zero homotopy type.\\

\begin{Prop}\label{energymeasure}
If $\m \in \cH_0(M)$ is such that $A_L(\m)=-c_u(H)$, then
$$\int E(x,v)\,d\m= c_u(H),$$ where $E(x,v)=\frac{\partial L}{\partial v} (x,v)\cdot v - L(x,v)$ is the energy.\\
\end{Prop}

\begin{Proof}
Let $\l\in \R$ and consider a new probability measure $\m_{\l}$ defined by
$$
\int f d\m_{\l} = \int f(x,\l v)d\m \qquad \forall \; f\in C^0_{\ell}(\rT M)\,.
$$
Clearly $\m_{\l}\in \cH_0(M)$. In fact,  since $\m\in \cH_0$, then there exist $\m_{\g_n}\rightarrow \m$ with $\g_n: [0,T_n]\rightarrow M$ contractible loops. Let us define $\m_{\g_n}^{\l}$ by
$$
\int f d\m_{\g_n}^{\l} =\frac{1}{T_n} \int_0^{T_n} f(\g_n(t),\l \dot{\g}_n(t)) dt = \frac{\l}{T_n} \int_0^{\frac{T_n}{\l}} f(x_n(t),\dot{x}_n(t)) dt \qquad
\forall\; f\in C^0_{\ell}(\rT M),
$$
where $x_n(t)=\g_n(\l t)$. Clearly $x_n$ are contractible and moreover $\m_{x_n}=\m^{\l}_{\g_n} \rightarrow \m_{\l}$ since
$$
\frac{1}{T_n}\int_0^{T_n} f(\g_n(t), \l \dot{\g}_n(t))\, dt \longrightarrow  
\int f(x,\l v)d\m = \int f d\m_{\l}\,.
$$

Now, set $F(\l) = \int L\,d\m_{\l} = \int L(x,\l v)\, d\m$ and observe that $F'(1)=0$ since $\m$ is action-minimizing in $\cH_0$. Moreover,
\beqano
0=F'(1) &=& \int \frac{\partial L}{\partial v}(x,v)\cdot v \,d\m = \\
&=& \int \left( E(x, v) + L(x,v) \right)\,d\m = \\
&=& \int E(x, v)\,d\m - c_u(H)\,.
\eeqano
\end{Proof}

This naturally raises:

\vspace{10 pt}

\noindent 
{\bf Question III}:  Is it true that action minimizing measures with zero homotopy are supported on the energy level
$\S_{c_u(H)}$? In other words: what is the energy of their ergodic components?\\

\vspace{10 pt}

An easy argument with the Tonelli theorem, to be supplied below during the proof of Theorem A, shows that there always exist minimizing invariant measures with zero homotopy which are supported on the energy level $\Sigma_{c_{u}(H)}$.

\begin{Proof}[Theorem A]
Let us assume by contradiction that $\S_{c_u(H)}$ is of virtual contact type. Then, using Legendre duality:
\begin{eqnarray*}
L(x,v)+c_u(H)\Big|_{\S_{c_{u}(H)}} &= &
L(x,v)+ E(x,v)\Big|_{\S_{c_{u}(H)}} \\
&=&  \frac{\partial L}{ \partial v}(x,v) \cdot v \\
&=&\Theta(X_E(x,v))\Big|_{\S_{c_{u}(H)}}
\end{eqnarray*}
where $X_E$ is the Euler-Lagrange vector field and
$\Theta$ is the pull-back of the canonical 1-form form $\lambda$ of $T^*M$ via the Legendre transform.

 Let $\widetilde{\S}_{c_u(H)} \stackrel{\pi}{\longrightarrow} {\S}_{c_u(H)}$ be the universal cover. Since $\dim M\geq 3$, this is the same as lifting everything to $\widetilde{M}$ and consider the energy level $c_{u}(H)$ of the lifted system (in fact, $E^{-1}(k) \longrightarrow \widetilde{M}$ is a sphere fibration over a simply connected manifold with simply connected fibers).

  We are assuming that there exists $\a$ on $\widetilde{\S}_{c_u(H)}$ such that
$d\widetilde{\Theta} = d\a$, with $\|\a\|_{C^0}<\infty$ and $\a(\widetilde{X}_E)\geq \e$. Moreover, since
$\pi_1(\widetilde{\S}_{c_u(H)})=0$, then there exists a smooth function $f:\widetilde{\S}_{c_u(H)}\longrightarrow \R$ such that
$$\widetilde{\Theta}= \a - df \qquad \mbox{on}\quad \widetilde{\S}_{c_u(H)}.$$
Extend $f$ to a smooth function on $\rT \widetilde{M}$. Then, $\widetilde{\Theta}+df$ is defined on all $\rT \widetilde{M}$ and has the property that $(\widetilde{\Theta}+df)(\widetilde{X}_E)\geq \e$ on 
$\widetilde{\S}_{c_u(H)}$.

\begin{Lem}\label{Lemmaintheproposition}
There exists $\d>0$ such that $\widetilde{\Theta}(\widetilde{X}_E) + df (\widetilde{X}_E) \geq \frac{\e}{2}$ on 
$E^{-1}\big( c_u(H)-\d,c_u(H)+\d\big)$.
\end{Lem}

We shall prove this lemma after completing the proof of Theorem A.

Let us now consider a sequence $\g_n: [0,T_n]\longrightarrow M$ of closed contractible Tonelli minimizers on $M$ (\ie each of them minimizes the action among contractible loops with the same time length) such that
$$
0 \leq \frac{1}{T_n}A_{L+c_u(H)}(\g_n) \longrightarrow 0.
$$
Each $\g_n$ has energy $k_n$. By a-priori compactness estimates \cite[Lemma 3.2.1]{CI99}, these $k_n$ are bounded. In fact, by superlinearity of $L$ we know that there exists $D>0$ such that  $L(x,v) \geq \|v\|-D$ for each $(x,v)\in \rT M$. Therefore,
$$
0 \longleftarrow A_{L+c_u(H)}(\g_n) \geq \frac{1}{T_n} \int_0^{T_n} \left(\|\dot{\g}_n\|-D+c_u(H) \right)\,.
$$
Applying the mean value theorem, we conclude that there are $t_0^n \in [0,T_n]$ such that
$\|\dot{\g}_n(t_0^n)\|\leq K$ uniformly. Hence, using the fact that the energy is constant along the orbits of the flow, we can conclude that the sequence $k_n$ is bounded.
By passing to a subsequence, if necessary, we can assume that $k_n\rightarrow k$ and
$$
\lim_{n\rightarrow +\infty} \frac{1}{T_n}A_L(\g_n)=-c_u(H).
$$
Let $\m \in \cH_0$ be a weak$^*$-limit of $\m_{\g_n}$. Clearly $\supp \m \subset E^{-1}(k)$ and $A_L(\m)=-c_u(H)$. By Proposition \ref{energymeasure} if follows that $k=c_u(H)$, so $k_n\rightarrow c_u(H)$.

\begin{Rem} Observe that we have proved the existence of invariant minimizing measures with
zero homotopy contained in the energy level $c_{u}(H)$.
\end{Rem}

For $n$ sufficiently large, $\G_n(t):= (\g_n(t),\dot{\g}_n(t)) \in E^{-1}(c_u(H)-\d, c_u(H)+\d)$. Thus, 
\beqano
\frac{1}{T_n} A_{L+c_u(H)}(\g_n) &=& \frac{1}{T_n} A_{L+k_n}(\g_n) + (c_u(H)-k_n) \\
&=& \frac{1}{T_n} \int_{\G_n}\Theta(X_E) + (c_u(H)-k_n).
\eeqano
Using Lemma \ref{Lemmaintheproposition}, we obtain
$$
\frac{1}{T_n} A_{L+c_u(H)}(\g_n) \geq \frac{\e}{2} + (c_u(H) - k_n).
$$
Taking the limit as $n$ goes to $+\infty$, we obtain a contradiction:
$$
0= \lim_{n\rightarrow +\infty} \frac{1}{T_n} A_{L+c_u(H)}(\g_n) \geq \lim_{n\rightarrow +\infty}\left(\frac{\e}{2} + (c_u(H) - k_n)\right) = \frac{\e}{2}.
$$
\end{Proof}

Let us now prove Lemma \ref{Lemmaintheproposition}.\\

\noindent {\bf Proof [Lemma \ref{Lemmaintheproposition}].} 
We shall choose $\d$ later. For the moment, let us consider a neighborhood of $\S_{c_u(H)}$ so that the following map is well defined
%of the form $\S_{c_u(H)} \times (-\d,\d)$ and the  map:
$$
\pi:  E^{-1}\big(c_u(H)-\d, c_u(H)+\d\big) \longrightarrow \S_{c_u(H)} \qquad {\rm s.t.}\quad \pi|E^{-1}(c_u(H)) = {\rm Id}.
$$
Lift this to $\rT \widetilde{M}$ and denote it by $\tilde{\pi}$. Let us extend 
$f: \widetilde{\S}_{c_u(H)}\longrightarrow \R$ by considering $f\circ \tilde{\pi}$. We want to show that we can choose $\d>0$ such that:
$$
\widetilde{\Theta}(\widetilde{X}_E) + df(d\tilde{\pi}(\widetilde{X}_E)) \geq \frac{\e}{2} \qquad{\rm on}\;
E^{-1}(c_u(H)-\d,c_u(H)+\d ).
$$
For this, just note that there exists $\d>0$ such that:
\beqano
&& \left| \widetilde{\Theta}(\widetilde{X}_E)(x,v) - \widetilde{\Theta}(\widetilde{X}_E)(\tilde{\pi}(x,v)) \right| \leq \frac{\e}{10}\\
&& \left| df_{\tilde{\pi}(x,v)} (d\tilde{\pi}(\widetilde{X}_E))(x,v)  -
df_{\tilde{\pi}(x,v)} (\widetilde{X}_E)(\pi(x,v))
 \right| \leq \frac{\e}{10}
\eeqano
because $|df|_{C^0} \leq K$, for some $K>0$, from the definition of being virtually contact.

Thus for $(x,v)\in \widetilde{E}^{-1}\big(c_u(H)-\d,c_u(H)+\d\big)$ we have:
\beqano
 \left(\widetilde{\Theta}(\widetilde{X}_E)+ df(d\tilde{\pi}(\widetilde{X}_E) \right)(x,v) &=&
\widetilde{\Theta}(\widetilde{X}_E)(x,v) - \widetilde{\Theta}(\widetilde{X}_E)(\tilde{\pi}(x,v)) \\
&& \ +
df_{\tilde{\pi}(x,v)} (d\tilde{\pi}(\widetilde{X}_E))(x,v) \\
&& -\;  df_{\tilde{\pi}(x,v)} (\widetilde{X}_E)(\pi(x,v)) \\
&& +\;\widetilde{\Theta}(\widetilde{X}_E) + df (\widetilde{X}_E)(\pi(x,v))  \\
&\geq& -\frac{\e}{10} -\frac{\e}{10} + \e = \frac{8\e}{10} > \frac{\e}{2}\,.
 \eeqano
\qed

\section{Symplectic invariance of $c_u(H)$} \label{sec3}

In this section we wish to prove the first item in Theorem B.

\begin{Prop}\label{thm1}
If $\Psi: \rT^*M \longrightarrow \rT^*M$ is a symplectomorphism such that $H\circ \Psi^{-1}$ is still of Tonelli type, then $c_u(H) = c_u(H\circ \Psi^{-1})$.
\end{Prop}

\begin{Proof}
Let $k<c_u(H)$. It follows from the definition of $c_u(H)$ (see subsection \ref{sec1.2}) that there exists a closed absolutely continuous contractible curve $\s:[0,T] \rightarrow M$ such that $A_{L+k}(\s)<0$. A-priori, this curve might not be an orbit, but applying Tonelli's theorem \cite{Mane91, Mat91} we can find a closed contractible minimizer (and hence an orbit for the flow) $\g: [0,T] \rightarrow M$ such that
$$ A_{L+k}(\g) \leq A_{L+k}(\s)<0.$$
Let $\G:= \left(\g, \frac{\partial L}{\partial v} (\g,\dot \g)\right)$ denote the image of this orbit in $\rT^*M$. Clearly, $\G$ is a flow line associated to the Hamiltonian vector field of $H$.
Since $\Psi$ is a symplectomorphism, then $\Psi(\G)$ will be a flow line for $H\circ \Psi^{-1}$. If we denote by $\g':=\pi \big( \Psi(\G) \big)$, where $\pi: \rT^* M \rightarrow M$ is the canonical projection, then $(\g',\dot{\g}')$ is an orbit of the Euler-Lagrange flow of $L'$, where $L'$ is the Lagrangian associated to $H\circ \Psi^{-1}$, and 
$$
\Psi(\G) = \left( \g', \frac{\partial L'}{\partial v} (\g',\dot \g')\right).
$$

\noindent {\it Claim.} $A_{L+k}(\g) = A_{L'+k}(\g')$.\\

In fact, using Fenchel-Legendre duality and the fact that the Hamiltonian is constant along the flow lines, we obtain (recall that $\omega =d\lambda$):
\beqano
A_{L'+k}(\g') &=& \int_0^T \big[L'(\g',\dot{\g}') + k \big]\,dt  \\
&=& \int_{\Psi(\G)} \l  + \int_0^T\big[ k- H'(\Psi(\G)(t))   \big] \,dt  \\
&= & \int_{\G} \Psi^*\l  + \int_0^T \big[ k- H(\G(t))\big]\,dt  \\
&=& \int_{\G} (\Psi^*\l-\l)  + \int_{\G} \l + \int_0^T\big[ k - H(\G(t))\big]\,dt  \\
&=& \int_{\G} (\Psi^*\l-\l)  + \int_0^T \big[ L(\g,\dot{\g}) + k\big]\,dt  \\
&=& \int_{\G} (\Psi^*\l-\l) + A_{L+k}(\g)\,.
\eeqano
Observe now that $\Psi$ is a symplectomorphism, \ie $\Psi^*\omega = \omega$, therefore $\Psi^*\l-\l$ is a closed $1$-form. Hence, using that $\G$ is contractible, we can conclude that $\int_{\G} (\Psi^*\l-\l)=0$. 
This concludes the proof of the claim.\\

It follows now from the definition of $c_u$, that $c_u(H\circ\Psi^{-1}) \geq c_u(H)$. The reversed inequality is proved in the same way, using that $\Psi$ is invertible.\\
\end{Proof}

%%%%%%%%%%%%%%%%%%%%%%%%%%%%%%%%%%%%%%%%%%%%%%
%%%%%%%%%%%%%%%%%%%%%%%%%%%%%%%%%%%%%%%%%%%
\section{Action-minimizing properties in the universal cover}\label{sec4}

In this section we would like to study the action-minimizing properties of the energy level corresponding to the Ma\~n\'e critical value of the universal cover. We refer the reader to \cite{CI99, Fathibook, SorLectNotes} for a more comprehensive introduction on Aubry--Mather--Ma\~n\'e--Fathi theory in the classical (compact) setting. The main difference (and difficulty) in this context is that the lift of the energy level to the universal cover is non-compact anymore (see also \cite{FathiMaderna}).

\subsection{Peierls barrier}
Let us denote by $\widetilde{H}$ the lift of the Hamiltonian to the universal cover, namely
$\widetilde{H}:= H\circ d\Pi_u$. Similary, $\widetilde{L}$ will represent the associated Lagrangian.
As done by Mather \cite{Mat93}, we want to study the action-minimizing properties of this system and define the so-called {\it Peierls barrier} associated to $\widetilde{H}$.\\
Given $\tx_1,\tx_2 \in \widetilde{M}$ and $T>0$, we define:
\beqa{finitetimepotential}
h^T_{\widetilde{H}}(\tx_1,\tx_2) : = \inf \int_0^T \widetilde{L}(\g(t),\dot{\g}(t))\,dt,
\eeqa
where the infimum is over all absolutely continuous curves $\g:[0,T]\rightarrow \widetilde{M}$ such that
$\g(0)=\tx_1$ and $\g(T)=\tx_2$. It follows from Tonelli theorem (it holds also in the non-compact case, assuming that $L$ is superlinear \cite{CI99}) that this infimum is actually a minimum.
The {\it Peierls barrier} is defined as follows:
\beqa{Peierlsbarrier}
h_{\widetilde{H}}(\tx_1,\tx_2) := \liminf_{T\rightarrow +\infty} \big[ h^T_{\widetilde{H}}(\tx_1,\tx_2) + c_u(H)T\big]\,.
\eeqa
One can check that $h_{\widetilde{H}}(\tx_1,\tx_2)>-\infty$ for any $\tx_1,\tx_2 \in \widetilde{M}$. 
Whilst in the compact case this quantity is always finite \cite{Mat93, Fathibook}, in the non-compact case it may be infinite.
However, it is easy to check that if this barrier is either finite everywhere, or identically equal to $+\infty$.

\begin{Lem}\label{lemmafinitepeierls}
If there exist $\tx_1,\tx_2 \in \widetilde{M}$ such that $h_{\widetilde{H}}(\tx_1,\tx_2)<+\infty$, then $h_{\widetilde{H}}$ is finite everywhere.
\end{Lem}

\begin{Proof}
Let $\tz_1, \tz_2 \in \widetilde{M}$. Consider the shortest unit-speed geodesics connecting $\tz_1$ to $\tx_1$ and $\tx_2$ to $\tz_2$, and denote them respectively by $\s_{i}$ and $\s_{f}$. Moreover, let
$$
A:= \sup_{\|\tilde{v}\|=1, \tx\in \widetilde{M}} \widetilde{L}(\tx,\tilde{v}) = 
\sup_{\|{v}\|=1, x\in {M}} {L}(x, {v}) < + \infty\,.
$$
Since $h_{\widetilde{H}}(\tx_1,\tx_2)<+\infty$, then there exist $\g_n:[0,T_n]\rightarrow \widetilde{M}$, with $T_n\rightarrow +\infty$, such that $\g_n(0)=\tilde{x}_1$, $\g_n(T_n)=\tx_2$ and 
$$
A_{\widetilde{L}+c_u(H)}(\g_n) \longrightarrow h_{\widetilde{H}}(\tx_1,\tx_2)\qquad {\rm as}\; n\rightarrow +\infty\,.
$$
Let $\tilde{T}_n:= T_n+ d(\tx_1,\tz_1) + d(\tx_2,\tz_2)$. We define new curves
$\s_n: [0,\tilde{T}_n] \rightarrow \widetilde{M}$ connecting $\tz_1$ to $\tz_2$, by $\s_n:=\s_i*\g_n*\s_f$.
Then:
\beqano
h_{\widetilde{H}}^{\widetilde{T}_n}(\tz_1,\tz_2)+ c_u(H)\widetilde{T}_n &\leq& A_{\widetilde{L}+c_u(H)}(\s_n)  \\
&=& A_{\widetilde{L}+c_u(H)}(\s_i) + A_{\widetilde{L}+c_u(H)}(\g_n)  + A_{\widetilde{L}+c_u(H)}(\s_f) \\
&\leq& A_{\widetilde{L}+c_u(H)}(\g_n) + [A+c_u(H)]\big(d(\tx_1,\tz_1) + d(\tx_2,\tz_2)\big)\,.
\eeqano
Therefore:
\beqano
h_{\widetilde{H}}(\tz_1,\tz_2) &\leq& \liminf_{n\rightarrow \infty}  \left( h_{\widetilde{H}}^{\widetilde{T}_n}(\tz_1,\tz_2)+ c_u(H)\widetilde{T}_n \right) \\
&\leq& \liminf_{n\rightarrow \infty} \left( A_{\widetilde{L}+c_u(H)}(\g_n) + [A+c_u(H)]\big(d(\tx_1,\tz_1) + d(\tx_2,\tz_2)\big)\right)  \\
&=& h_{\widetilde{H}}(\tx_1,\tx_2) + [A+c_u(H)]\big(d(\tx_1,\tz_1) + d(\tx_2,\tz_2)\big) <\infty.
\eeqano
\end{Proof}

\begin{Rem}\label{rem2}
It is not difficult to construct an example of a Tonelli Lagrangian on a non-compact manifold, whose Peierls barrier is identically $+\infty$. For example (see \cite{CI99} for more details) one could consider
\begin{eqnarray*}
L: \rT \R^2 & \longrightarrow& \R\\
(x,v) &\longmapsto& \frac{1}{2}\|v\|^2 + U(x)
\end{eqnarray*}
where $\|\cdot\|$ denotes the euclidean norm on $\R^2$ and $U(x)$ is a smooth function such that $U(x) \geq 0$ for all $x$, $U(x)=\frac{1}{\|x\|}$ for $\|x\|\geq 2$ and $U(x)=2$ for $0\leq \|x\| \leq 1$.
However this Lagrangian is not a lift of a Lagrangian on a closed surface (see also Question II in the Introduction).\\
%It is possible to check that $c_u(L)=0$. In fact, obviously  $c_u(L) \leq 0$ since $L\geq 0$; on the other hand if $\g_n$ is a smooth closed curve with length $\ell(\g_n)=1$, $\|\g_n(t)\| \geq n$ and energy $E(\g_n)= \frac{1}{2}\|\dot{\g}_n\|^2 - U(\g_n) \equiv 0$, then:
%\begin{eqnarray*}
%c_u(L) &\geq& - \inf_{n>0} A_L(\g_n) = - \int_0^{T_n} \left( \frac{1}{2}\|\dot{\g}_n\|^2 + U(\g_n)
%\right)\,dt =\\
%&=& - \int_0^{\frac{1}{\|\dot{\g}_n\|}} \|\dot{\g}_n\|^2 \,dt = -\|\dot{\g}\|
%\end{eqnarray*}

\end{Rem}

%%%%%%%%%%%%%%%%%%%%%%%%%%%%%%%%%%%%%%%%%%%%%%%%%%%%%

Let us prove that the property of having a finite Peierls barrier is somehow simplectically invariant. Namely, 
if $\Psi: \rT^*M \longrightarrow \rT^*M$ is a symplectomorphism such that $H':=H\circ \Psi^{-1}$ is still of Tonelli type,
and we denote by  $\widetilde{H'}$ and $\widetilde{L'}$ respectively the lifts of this Hamiltonian and the corresponding Lagrangian to the universal cover, 
% Clearly, $\widetilde{H'}=\widetilde{H} \circ \widetilde{\Psi}$, where $\widetilde{\Psi}$ is the lift of the symplectomorphism $\Psi$.
then the following is true.
% Moreover, let us denote by $h_{\widetilde{H}'}$ the Peierls barrier associated to $\widetilde{H'}$ and $\widetilde{L'}$. Then, we shall prove the following.

\begin{Theo}\label{thm2}
$h_{\widetilde{H}}$ is finite if and only if $h_{\widetilde{H}'}$ is finite.
\end{Theo}

\begin{Proof}
It suffices to prove that $h_{\widetilde{H}}$ being finite implies that $h_{\widetilde{H}'}$ is also finite. Using the invertibility of ${\Psi}$, one can similarly prove the other implication.\\
Suppose that there exist $\tx_1,\tx_2 \in \widetilde{M}$ such that $h_{\widetilde{H}}(\tx_1,\tx_2)<+\infty$. Then, we can find $\g_n:[0,T_n]\rightarrow \widetilde{M}$, with $T_n\rightarrow +\infty$, such that $\g_n(0)=\tilde{x}_1$, $\g_n(T_n)=\tx_2$ and 
\beqa{formula1}
A_{\widetilde{L}+c_u(H)}(\g_n) \longrightarrow h_{\widetilde{H}}(\tx_1,\tx_2)\qquad {\rm as}\; n\rightarrow +\infty\,.
\eeqa
Up to applying Tonelli theorem, we can assume that each $\g_n$ is a  Tonelli minimizer and hence an orbit of the flow. In particular, since they have bounded actions, using \cite[Lemma 3-2.1]{CI99} one can deduce that their velocities are bounded, \ie $\|\dot{\g}_n(t)\| \leq C $ for all $n$ and for all $t\in [0,T_n].$ 
Let us now consider the corresponding flow lines for the lifted Hamiltonian:
$$
\G_n(t):= \left(\g_n(t), \frac{\partial L}{\partial v}(\g_n(t),\dot{\g}_n(t)) \right).
$$
Since $\widetilde{\Psi}$ is a symplectomorphism, then $\G'_n:=\widetilde{\Psi}(\G_n)$ is still an orbit of $\widetilde{H'}$. In particular, since $\dot{\g}_n(0)$ and $\dot{\g}_n(T_n)$ all lie in a compact region, then the endpoints $\G_n'(0)$ and $\G_n'(T_n)$ lie in a compact region of $\rT^*\widetilde{M}$. Therefore, up to extracting a convergent subsequence, we can assume that the endpoints of the projected curves $\g'_n:= \pi \G'_n$ converge as $n\rightarrow +\infty$:
\beqa{formula2}
\g_n'(0) \rightarrow \tz_1 \qquad {\rm and} \qquad \g_n'(T_n) \rightarrow \tz_2\,.
\eeqa
Let us now consider  the shortest unit-speed geodesics connecting $\tz_1$ to $\g'_n(0)$ and $\g'_n(T_n)$ to $\tz_2$, denoting them respectively $\s_{n,i}$ and $\s_{n,f}$.
Let $\tilde{T}_n:= T_n+ d(\g'_n(0),\tz_1) + d(\g'_n(T_n),\tz_2)$ and  define a new sequence of curves
$\s_n: [0,\tilde{T}_n] \rightarrow \widetilde{M}$ connecting $\tz_1$ to $\tz_2$, by $\s_n:=\s_{n,i}*\g'_n*\s_{n,f}$.
Then (using also Proposition \ref{thm1}):
\beqano
h_{\widetilde{H}'}^{\widetilde{T}_n}(\tz_1,\tz_2)+ c_u(H')\widetilde{T}_n &\leq& A_{\widetilde{L}'+c_u(H')}(\s_n)  \\
&=& A_{\widetilde{L}'+c_u(H')}(\s_{n,i}) + A_{\widetilde{L}'+c_u(H')}(\g'_n)  + A_{\widetilde{L}'+c_u(H')}(\s_{n,f}) \\
&\leq& A_{\widetilde{L}'+c_u(H')}(\g'_n) + [A+c_u(H')]\big(d(\g'_n(0),\tz_1) + d(\g'_n(T_n),\tz_2)\big)\,,
\eeqano
where $$
A:= \sup_{\|\tilde{v}\|=1, \tx\in \widetilde{M}} \widetilde{L}(\tx,\tilde{v}) = 
\sup_{\|{v}\|=1, x\in {M}} {L}(x, {v}) < + \infty\,.
$$

Let $\widetilde{\l}$ denote the lift of the Liouville form $\l$ to $\rT^*\widetilde{M}$. Observe that since $\widetilde{\Psi}$ is a symplectomorphism, then $\widetilde{\Psi}^*\widetilde{\l}-\widetilde{\l} = d\widetilde{F}$ for some $F:\rT^*\widetilde{M}\rightarrow \R$.
We want to show that 
$$
A_{\widetilde{L}'+c_u(H')}(\g'_n) = A_{\widetilde{L}+c_u(H)}(\g_n) + F(\G_n(T_n)) - F(\G_n(0))\,.
$$

Using that the Hamiltonian is constant along the orbits and that $\G_n(T_n)$ and $\G_n(0)$ all lie in the same compact region, we obtain:
\beqano
A_{\widetilde{L}'+c_u(H')}(\g'_n) &=& \int_0^{T_n} \left(\widetilde{L}'(\g_n',\dot{\g_n}') + c_u(H')\right)\,dt  \\
&=& \int_{\G'_n} \widetilde{\l}  + \int_0^{T_n}\left(c_u(H')- \widetilde{H}'(\G_n'(t))\right) \,dt  \\
&= & \int_{\G_n} \widetilde{\Psi}^*\widetilde{\l}  + \int_0^{T_n} \left(c_u(H)- \widetilde{H}(\G_n(t))\right)\,dt  \\
&=& \int_{\G_n} (\widetilde{\Psi}^*\widetilde{\l}-\widetilde{\l})  + \int_{\G_n} \widetilde{\l} + \int_0^{T_n}\left(c_u(H) - \widetilde{H}(\G_n(t))\right)\,dt  \\
&=& \int_{\G_n} (\widetilde{\Psi}^*\widetilde{\l}-\widetilde{\l})  + \int_0^{T_n} \left(\widetilde{L}(\g_n,\dot{\g}_n) + c_u(H)\right)\,dt \\
&=& \int_{\G_n} (\widetilde{\Psi}^*\widetilde{\l}-\widetilde{\l}) + A_{\widetilde{L}+c_u(H)}(\g_n)  \\
&=& A_{\widetilde{L}+c_u(H)}(\g_n) + F(\G_n(T_n)) - F(\G_n(0)) \\
&\leq& A_{\widetilde{L}+c_u(H)}(\g_n) + const. \\  
\eeqano

Then, using  (\ref{formula1}),  (\ref{formula2}) and Proposition \ref{thm1}, we can conclude:
\beqano
h_{\widetilde{H}'}(\tz_1,\tz_2) &\leq& \liminf_{n\rightarrow \infty} \left( h_{\widetilde{H}'}^{\widetilde{T}_n}(\tz_1,\tz_2)+ c_u(H')\widetilde{T}_n\right) \\
&=& \liminf_{n\rightarrow \infty} \left(A_{\widetilde{L}'+c_u(H')}(\g'_n) + [A+c_u(H')]\big(d(\g'_n(0),\tz_1) + d(\g'_n(T_n),\tz_2)\big) \right) \\
&\leq& \liminf_{n\rightarrow \infty} \left(A_{\widetilde{L}+c_u(H)}(\g_n) +  {const} \right)  \\
&=& h_{\widetilde{H}}(\tx_1,\tx_2) + { const} < +\infty.
\eeqano
\end{Proof}

%%%%%%%%%%%%%%%%%%%%%%%%%%%%%%%%%%%%%%%%%%%%%%%%%%%%%%%%%%%%%%%%%%%%%%%
%%%%%%%%%%%%%%%%%%%%%%%%%%%%%%%%%%%%%%%%%%%%%%%%%%%%%%%%%%%%%%%%%%%%%%%
\subsection{Aubry set}
In the study of the dynamics of the system and its action-minimizing properties, a very important r\^ole is played by the set in which the Peierls barrier vanishes: 
$$
\cA_{\widetilde{H}} := \{\tx\in \widetilde{M}:\; h_{\widetilde{H}}(\tx,\tx)=0\}\,.
$$
This set is usually called {\it projected Aubry set}  (or  {\it Peierls set}) of $\widetilde{H}$ (see \cite{Mat93,Mane91,Fathibook}).

It is important to point out that while in the compact case this set is always non-empty, in the non-compact case   $\cA_{\widetilde{H}}$ might be empty, even if the Peierls barrier is finite, as it is shown in Section \ref{examples}.

\begin{Rem} It is straightforward to check the following behaviour under coverings $p:M_{1}\to M_{2}$:
$c_{1}\leq c_{2}$, $h^{1}\geq h^{2}$ and $p\cA^{1}\subset \cA^2$, where $c_i$, $h^i$ and $\cA^i$ are respectively the critical value, the Peierls barrier and the Aubry set in $M_i$.
\end{Rem}

We want to show that being empty or not, is also a symplectic invariant property of the system, in the same sense as we explained above.

\begin{Theo}\label{thm3}
Let $\Psi: \rT^*M \longrightarrow \rT^*M$ be a symplectomorphism such that $H':=H\circ \Psi^{-1}$ is still of Tonelli type. Then, $\cA_{\widetilde{H}} \neq \emptyset$  if and only if $\cA_{\widetilde{H}'} \neq \emptyset$.
\end{Theo}

\begin{Proof}
We shall prove that $\cA_{\widetilde{H}} \neq \emptyset$  implies  $\cA_{\widetilde{H}'} \neq \emptyset$.
Using the invertibility of ${\Psi}$, one can similarly prove the other implication. 
Suppose that $\tx\in \cA_{\widetilde{H}}$, \ie there exists $\tx \in \widetilde{M}$ such that $h_{\widetilde{H}}(\tx,\tx)=0$. 
Then, we can find closed loops $\a_n:[0,T_n]\rightarrow \widetilde{M}$, with $T_n\rightarrow +\infty$, such that $\a_n(0)=\tilde{x}$ and 
\beqano
A_{\widetilde{L}+c_u(H)}(\a_n) \longrightarrow 0\qquad {\rm as}\; n\rightarrow +\infty\,.
\eeqano
These curves $\a_n$ are not necessarily closed orbits. Therefore, for any given $T_n$, we can apply Tonelli Theorem for closed contractible loops in $M$ with period $T_n$ and obtain closed orbits that we can lift to $\widetilde{M}$. We shall denote these new orbits by 
$\g_n:[0,T_n]\rightarrow \widetilde{M}$. Observe that since they are Tonelli minimizers in their respective class, then
\beqa{formula3}
0\leq A_{\widetilde{L}+c_u(H)}(\g_n) \leq A_{\widetilde{L}+c_u(H)}(\a_n)   \longrightarrow 0\qquad {\rm as}\; n\rightarrow +\infty\,.
\eeqa
It is important to notice that it is not true anymore that $\g_n(0)=\tilde{x}$, but we can nevertheless assume that these ``end-points'' are all contained in a compact region of $\widetilde{M}$. 
In particular, since they have bounded actions, using \cite[Lemma 3-2.1]{CI99} one can deduce that their velocities are bounded, \ie $\|\dot{\g}_n(t)\| \leq C $ for all $n$ and for all $t\in [0,T_n].$ 
Let us now consider the corresponding flow lines for the lifted Hamiltonian:
$$
\G_n(t):= \left(\g_n(t), \frac{\partial \widetilde{L}}{\partial v}(\g_n(t),\dot{\g}_n(t)) \right).
$$
Since $\widetilde{\Psi}$ is a symplectomorphism, then $\G'_n:=\widetilde{\Psi}(\G_n)$ is still a closed orbit of $\widetilde{H'}$. In particular, since all $\dot{\g}_n(0)$ lie in a compact region, then all $\G_n'(0)$  lie in a compact region of $\rT^*\widetilde{M}$. Therefore, up to extracting a convergent subsequence, we can assume that 
\beqa{formula4}
\g_n'(0) \rightarrow \tz \qquad {\rm as}\; n\rightarrow +\infty\,,
\eeqa
where $\g'_n:= \pi \G'_n$.
Let us now consider the  shortest unit-speed geodesics connecting $\tz$ to $\g'_n(0)$ and $\g'_n(0)$ to $\tz$, denoting them respectively $\s_{n,i}$ and $\s_{n,f}$.
Let $\tilde{T}_n:= T_n+ 2d(\g'_n(0),\tz) $ and  define a new sequence of closed curves
$\s_n: [0,\tilde{T}_n] \rightarrow \widetilde{M}$, by $\s_n:=\s_{n,i}*\g'_n*\s_{n,f}$.
Then (using also Proposition \ref{thm1}):
\beqano
h_{\widetilde{H}'}^{\widetilde{T}_n}(\tz,\tz)+ c_u(H')\widetilde{T}_n &\leq& A_{\widetilde{L}'+c_u(H')}(\s_n) \\
&=& A_{\widetilde{L}'+c_u(H')}(\s_{n,i}) + A_{\widetilde{L}'+c_u(H')}(\g'_n)  + A_{\widetilde{L}'+c_u(H')}(\s_{n,f}) \\
&\leq& A_{\widetilde{L}'+c_u(H')}(\g'_n) + 2[A+c_u(H')] d(\g'_n(0),\tz)\,,
\eeqano
where 
$$
A:= \sup_{\|\tilde{v}\|=1, \tx\in \widetilde{M}} \widetilde{L}(\tx,\tilde{v}) = 
\sup_{\|{v}\|=1, x\in {M}} {L}(x, {v}) < + \infty\,.
$$

If as before  $\widetilde{\l}$ denotes the lift of the Liouville form $\l$ to $\rT^*\widetilde{M}$ and 
$F:\rT^*\widetilde{M}\rightarrow \R$ is such that $\widetilde{\Psi}^*\widetilde{\l}-\widetilde{\l} = d\widetilde{F}$, then we want to prove that:
$$
A_{\widetilde{L}'+c_u(H')}(\g'_n) = A_{\widetilde{L}+c_u(H)}(\g_n)\,.
$$

Using in fact that the Hamiltonian is constant along the orbits, we obtain:
\beqano
A_{\widetilde{L}'+c_u(H')}(\g'_n) &=& \int_0^{T_n} \left(\widetilde{L}'(\g_n',\dot{\g_n}') + c_u(H')\right)\,dt  \\
&=& \int_{\G'_n} \widetilde{\l}  + \int_0^{T_n}\left(c_u(H')- \widetilde{H}'(\G_n'(t))\right) \,dt  \\
&= & \int_{\G_n} \widetilde{\Psi}^*\widetilde{\l}  + \int_0^{T_n} \left(c_u(H)- \widetilde{H}(\G_n(t))\right)\,dt  \\
&=& \int_{\G_n} (\widetilde{\Psi}^*\widetilde{\l}-\widetilde{\l})  + \int_{\G_n} \widetilde{\l} + \int_0^{T_n}\left(c_u(H) - \widetilde{H}(\G_n(t))\right)\,dt  \\
&=& \int_{\G_n} (\widetilde{\Psi}^*\widetilde{\l}-\widetilde{\l})  + \int_0^{T_n} \left(\widetilde{L}(\g_n,\dot{\g}_n) + c_u(H)\right)\,dt  \\
&=& \int_{\G_n} d\widetilde{F} + A_{\widetilde{L}+c_u(H)}(\g_n)  \\
&=& A_{\widetilde{L}+c_u(H)}(\g_n). \\  
\eeqano

Then, using  (\ref{formula3}),  (\ref{formula4}) and Proposition \ref{thm1}, we can conclude:
\beqano
0\leq h_{\widetilde{H}'}(\tz,\tz) &\leq& \liminf_{n\rightarrow \infty} \left(h_{\widetilde{H}'}^{\widetilde{T}_n}(\tz,\tz)+ c_u(H')\widetilde{T}_n\right) \\
&=& \liminf_{n\rightarrow \infty} \left(A_{\widetilde{L}'+c_u(H')}(\g'_n) + 2\,[A+c_u(H')] d(\g'_n(0),\tz) \right) \\
&=& \liminf_{n\rightarrow \infty} A_{\widetilde{L}+c_u(H)}(\g_n)  = 0\,.
\eeqano

Therefore, $\tz\in \cA_{\widetilde{H}'}$.

\end{Proof}

This concludes the proof of the first three items in Theorem B.\\

%%%%%%%%%%%%%%%%%%%%%%%
\vspace{10 pt}

The projected Aubry set is closely related to the existence and the properties of viscosity solutions and subsolutions of Hamilton-Jacobi equation, as pointed out by Fathi \cite{Fathibook}. Let us recall the notion of viscosity subsolution, supersolution and solution. \\

Let $U$ be an open set of $M$. We shall say that a continuous function $u: U\longrightarrow \R$ is a {\it viscosity subsolution} (resp. {\it supersolution}) of $\widetilde{H}(x,d_xu)=k$, if for each $C^1$ function $\phi:U\longrightarrow \R$ satisfying $\phi\geq u$ (resp. $\phi\leq u$), and each point $x_0\in U$ satisfying $\phi(x_0)=u(x_0)$, we have $\widetilde{H}(x,d_{x_0}\phi) \leq k$ (resp. $\widetilde{H}(x,d_{x_0}\phi) \geq k$). 
A function is a {\it viscosity solution} if it is both a viscosity subsolution and a viscosity supersolution.
Observe, that since $\widetilde{H}(x,p)$ is convex and superlinear in $p$, it is well-know (see for instance \cite{Barles}), that a function $u:U\longrightarrow \R$ is a viscosity subsolution of $\widetilde{H}(x,d_xu)=k$ if and only if it is Lipschitz and $\widetilde{H}(x,d_xu)\leq k$ almost everywhere in $U$.

\begin{Rem}
It turns out that $c_u(H)$ is the infimum of the $k$'s for which $\widetilde{H}(x,d_xu)=k$ admits a subsolution and the unique energy value in which viscosity solutions exist \cite{CIPP98, Fathibook}. In particular, a subsolution corresponding to $k=c_u(H)$ is often called {\it critical subsolution}.
\end{Rem}

The relation between subsolutions and the projected Aubry set is explained by the following result.

\begin{Theo}[{\bf Fathi-Siconolfi \cite{FathiSiconolfi}}]\label{FathiSic}
For a point $x_0\in \widetilde{M}$, the following conditions are equivalent:
\begin{itemize}
\item[{\rm i)}] the point $x_0$ is in $\cA_{\widetilde{H}}$;
\item[{\rm ii)}] every critical subsolution is differentiable at $x_0$;
\item[{\rm iii)}] there does not exist a critical subsolution $u$ which is strict at $x_0$, \ie  $\widetilde{H}(x_0,d_{x_0}u)< c_u(H)$.\\
\end{itemize}
\end{Theo}

In particular, it is easy to deduce from the above theorem that on the projected Aubry set the differential of critical subsolutions is prescribed.\\

\begin{Cor}
All critical subsolutions have the same differential on the projected Aubry set.
\end{Cor}

\begin{Proof}
Suppose by contradiction that there exist two critical subsolutions $u_1$ and $u_2$, whose differentials do not coincide at some point $x_0\in \cA_{\widetilde{H}}$, \ie $d_{x_0}u_1 \neq d_{x_0}u_2$.
Using the  convexity in the fibers of $\widetilde{H}$, it is easy to check that $\frac{u_1 + u_2}{2}$ is also a critical subsolutions. However, using the fact the strict convexity of $H$ in the fibers, we deduce that:
$$
\widetilde{H}\left(x_0, d_{x_0}\left(\frac{u_1+ u_2}{2}\right)\right) <
\frac{1}{2}\widetilde{H}\left(x_0, d_{x_0}u_1\right) + \frac{1}{2}\widetilde{H}\left(x_0, d_{x_0}u_1\right) \leq c_u(H)\,.
$$
Therefore, $\frac{u_1 + u_2}{2}$ is a critical subsolution which is strict at $x_0\in \cA_{\widetilde{H}}$. This clearly contradicts item (iii) in Theorem \ref{FathiSic}.\\
\end{Proof}

Hence, if we denote by $\cS_{\widetilde{H}}$ the set of critical subsolutions of $\widetilde{H}(x,d_xu)=c_u(H)$, then we can consider the following intersection:

$$\cA^*_{\widetilde{H}} = \bigcap_{u\in \cS_{\widetilde{H}}} \{(x,d_xu):\; u\;\mbox{is differentiable at}\; x\}. $$

This set is what is usually called  the {\it Aubry set}. It follows from the above discussion that $\pi(\cA^*_{\widetilde{H}}) = \cA_{\widetilde{H}}$, where
$\pi: \rT^*\widetilde{M}\longrightarrow \widetilde{M}$ denotes the canonical projection. Therefore, it is a graph over the projected Aubry set and hence it is non-empty if and only if $\cA_{\widetilde{H}}$ is non-empty.  Moreover, if it is non-empty, it is closed (since the Peierls  barrier is continuous) and invariant. %(PROOF? IT ALSO FOLLOWS FROM THE SYMPLECTIC INVARIANCE). \\

It is possible to provide a better description of the Aubry set, just in terms of the Peierls  barrier. In fact, it was proved by Fathi \cite{Fathibook} that  if $h_{\widetilde{H}}$ is finite, then  for every $x\in M$, $h_{\widetilde{H}}^x(\cdot):= h_{\widetilde{H}}(x, \cdot)$ is a global viscosity solution of $\widetilde{H}(x,d_xu)=c_u(H)$.
Therefore, the Aubry set can be equivalently defined as:
$$\cA^*_{\widetilde{H}} = \left \{(x, \partial_2 h(x, x)):\quad x\in \cA_{\widetilde{H}} \right\}. $$ 

We shall prove that this set is symplectically invariant in Section \ref{barrierphasespace} (item (4) in Theorem B).\\

%%%%%%%%%%%%%%%%%%%%%%%%%%%%%%%%%%%%%%%%%%%%%%%

%%%%%%%%%%%%%%%%%%%%%%%%%%%%%%%%%%%%%%%%%%%%%%%%%%%%%%%%%%%%%%%%%%%%%%%%%%%%%%%%%%%%%%%%%%%%%%%%
%%%%%%%%%%%%%%%%%%%%%%%%%%%%%%%%%%%%%%%%%%%%%%%%

\section{Barrier in phase space}\label{barrierphasespace}\label{sec5}

In this section we prove that $\cA^*_{\widetilde{H}}$ is simplectically invariant thus completing the proof of Theorem B. In the compact case this has been proven by Patrick Bernard in \cite{B07}.  In the following, we adapt his approach to this setting and prove the following result.

\begin{Theo}\label{invarianceAubry}
If $\Psi: \rT^*M \longrightarrow \rT^*M$ is a symplectomorphism such that $H':=H\circ \Psi$ is still of Tonelli type, then $\cA^*_{\widetilde{H}'} = \widetilde{\Psi}^{-1}\left(\cA^*_{\widetilde{H}}\right).$
\end{Theo}

As before, $\tM$ denotes the universal cover of $M$ (compact connected smooth manifold) and $\tH: \rT^*\tM \longrightarrow \R$ is the lift of a Tonelli Hamiltonian $H: \rT^*M \longrightarrow \R$, while $\Phi^{\widetilde{H}}_t$ is its Hamiltonian flow.
As done in \cite{B07}, let us define a {\it Peierls barrier  in the phase space}. First of all, let us introduce the notion  of {\it pre-orbit}.\\

\begin{Def}[{\bf Pre-orbit}]\label{preorbit}
Given  $X_0, X_1 \in \rT^*\tM$, we say that a sequence of curves $\uY=(Y_n)$, where 
$Y_n: [0,T_n] \rightarrow \rT^*\tM$,  is a pre-orbit between $X_0$ and $X_1$ if:
\begin{itemize}
\item[{\rm i)}] for each $n$ the curve $Y_n$ has a finite number $N_n$ of discontinuities $T_n^i \in (0,T_n)$, such that $T_n^{i} < T_n^{i+1}$ for all $1\leq i \leq N_n$. Moreover, we shall denote $T_n^0:=0$ and $T_n^{N_n+1}:=T_n$.
\item[{\rm ii)}] For each $n$ and for each $s \in [0, T_n^{i+1}-T_n^i)$, $Y_n(T_n^i+s)=\Phi^{\widetilde{H}}_s(Y_n(T_n^i))$. We denote by  $Y_n(T_n^{i+1}-):= \Phi^{\widetilde{H}}_{T_n^{i+1}-T_n^i}(Y_n(T_n^i))$ and we ask that 
$Y_n(T_n)=Y_n(T_n^{N_n+1}-)$.
\item[{\rm iii)}] $T_n \longrightarrow +\infty$ as $n\rightarrow +\infty$.
\item[{\rm iv)}] $Y_n(0) \longrightarrow X_0$ and $Y_n(T_n) \longrightarrow X_1$ as $n\rightarrow +\infty$. Moreover, if we denote $\D(Y_n):= \sum_{i=1}^{N_n} \dist (Y_n(T^i_n-), Y_n(T^i_n) )$, then 
$\D(Y_n)\longrightarrow 0$ as $n\rightarrow +\infty$ ($\dist(\cdot, \cdot)$ denotes the distance on $T^*\widetilde{M}$ induced by the Riemannian metric on $M$).
\item[{\rm v)}] All curves $Y_n$ have equi-bounded energy, \ie there exists $K=K(\uY)\in \R$ such that
$\widetilde{H}(Y_n(s))\leq K$ for all $n$ and for all $s\in [0,T_n]$.\\
\end{itemize}
\end{Def}

\begin{Rem}
Observe that condition $({\rm v})$ in the above definition is different from the one in \cite{B07}: it is needed because of the lack of compactness.
\end{Rem}

Let us now define what we mean by {\it action of a pre-orbit}. Let $\uY=(Y_n)$ be a pre-orbit. Then, the action of $Y$ is given by:
$$
A_{\tH}(Y):= \liminf_{n\rightarrow \infty} A_{\tH}(Y_n),
$$ 
where
$$
A_{\tH}(Y_n):= \int_0^{T_n}   \big[ \widetilde{\l}_{Y_n(t)}\big(\dot{Y}_n(t)\big) - \tH(Y_n(t)) \big] \, dt\,.
$$

\begin{Lem}
If there exists a pre-orbit between $X_0$ and $X_1$, then $\tH(X_0)=\tH(X_1)$. 
\end{Lem}

\begin{Proof}
Let $\uY=(Y_n)$ be such a pre-orbit. Then, observing that $\tH(Y_n(T_n^{i})) = \tH(Y_n(T_n^{i+1}-))$, since the energy is constant along the orbits, we obtain:
\beqano
\tH(X_1)-\tH(X_0) &=& \lim_{n\rightarrow +\infty} \tH(Y_n(T_n)) - \tH(Y_n(0)) \\
&=& \lim_{n\rightarrow +\infty} \sum_{i=0}^{N_n} \tH(Y_n(T_n^{i+1})) - \tH(Y_n(T_n^{i}))   \\
&=&  \lim_{n\rightarrow +\infty} \sum_{i=0}^{N_n} \tH(Y_n(T_n^{i+1})) - \tH(Y_n(T_n^{i+1}-)) \\ 
&=& \lim_{n\rightarrow +\infty}  \sum_{i=1}^{N_n} \tH(Y_n(T_n^{i})) - \tH(Y_n(T_n^{i}-)),
\eeqano
where in the last equality we used that $Y_n(T_n^{N_n+1})=Y_n(T_n)=Y_n(T_n^{N_n+1}-)$.
Since these orbits have bounded energy, then the Hamiltonian $\tH$ will be Lipschitz in this region (the region itself is not compact, but the value of $\tH$ here depends only on the value of $H$ on the projection of this region  to $\rT^*M$, which is compact). Then, denoting this Lipschitz constant by $C$ we get:
\beqano
\tH(X_1)-\tH(X_0) &=& \lim_{n\rightarrow +\infty}  \sum_{i=1}^{N_n} \tH(Y_n(T_n^{i})) - \tH(Y_n(T_n^{i}-))  \\
&\leq& \lim_{n\rightarrow +\infty}  \sum_{i=1}^{N_n} C\, \dist (Y_n(T_n^{i}) , Y_n(T_n^{i}-) ) \\
&=&  \lim_{n\rightarrow +\infty}  C\, \D(Y_n) = 0\,. 
\eeqano
\end{Proof}

Let us now define the equivalent of {\it Peierls barrier} (see \cite{Mat93}), but in the phase space $\rT^*\tM$. Given $X_0,X_1 \in \rT^*\tM$ let us denote by $\cY_{\tH}(X_0,X_1)$ the set of pre-orbits between $X_0$ and $X_1$. Notice that this set is empty if $X_0$ and $X_1$ are not in the same energy level. Then:
\beqano
B_{\tH}: \rT^*\tM \times \rT^*\tM &\longrightarrow& \R\cup\{+\infty\}\\
(X_0,X_1) &\longmapsto& \inf_{\uY\in \cY_{\tH}(X_0,X_1)} A_{\tH}(\uY), 
\eeqano
setting $B_{\tH}(X_0,X_1)=+\infty$,  whenever $\cY_{\tH}(X_0,X_1)= \emptyset$.

\begin{Prop}\label{Prop2.5Bernard}
For each $t>0$ and $X_0,X_1\in \rT^*\widetilde{M}$, the following equality holds:
$$
B_{\widetilde{H}}(X_0,X_1)= B_{\widetilde{H}}(\Phi^{\widetilde{H}}_t(X_0),X_1) +
\int_0^t \left[ \widetilde{\lambda}_{\Phi^{\widetilde{H}}_s(X_0)}\big(X_{\widetilde{H}}(\Phi^{\widetilde{H}}_s(X_0))\big) - \widetilde{H}(\Phi^{\widetilde{H}}_s(X_0))  \right]\,ds
$$
and
$$
B_{\widetilde{H}}(X_0, \Phi^{\widetilde{H}}_t(X_1))  =  B_{\widetilde{H}}(X_0,X_1) +
\int_0^t \left[ \widetilde{\lambda}_{\Phi^{\widetilde{H}}_s(X_1)} \big(X_{\widetilde{H}}(\Phi^{\widetilde{H}}_s(X_1)) \big) - \widetilde{H}(\Phi^{\widetilde{H}}_s(X_1))  \right]\,ds,
$$
where $X_{\widetilde{H}}(\cdot)$ denotes the Hamiltonian vector field associated to $\widetilde{H}$ and $\widetilde{\lambda}$ the lift of the Liouville form to $\rT^*\widetilde{M}$.\\
\end{Prop}

\begin{Proof}
We only prove the first equality, since the second one can be proved similarly. To each pre-orbit $\uY$ between $X_0$ and $X_1$, we associate the pre-orbit $\uZ$ between $\Phi^{\widetilde{H}}_t(X_0)$ and $X_1$ defined by
\beqano
Z_n: [0, T_n-t] &\longrightarrow & \rT^*\widetilde{M}\\
s &\longmapsto& Y_n(s+t)\,.
\eeqano

We obtain:
$$
A_{\widetilde{H}}(\uY) = A_{\widetilde{H}}(\uZ) + 
\int_0^t \left[ \widetilde{\lambda}_{\Phi^{\widetilde{H}}_s(X_0)}   \big(X_{\widetilde{H}}(\Phi^{\widetilde{H}}_s(X_0))\big) - \widetilde{H}(\Phi^{\widetilde{H}}_s(X_0))\right]\,ds\,. 
$$
This implies that
$$
B_{\widetilde{H}}(\Phi^{\widetilde{H}}_t(X_0),X_1) \leq B_{\widetilde{H}}(X_0,X_1) -
\int_0^t \left[ \widetilde{\lambda}_{\Phi^{\widetilde{H}}_s(X_0)}\big(X_{\widetilde{H}}(\Phi^{\widetilde{H}}_s(X_0))\big) - \widetilde{H}(\Phi^{\widetilde{H}}_s(X_0))  \right]\,ds.
$$

In a similar way, we associate to each pre-orbit $\uZ$ between $\Phi^{\widetilde{H}}(X_0)$ and $X_1$, the pre-orbit $\uY$ between $X_0$ and $X_1$ defined by 
$Y_n(s):= \Phi^{\widetilde{H}}_{s-t}(Z_n(0))$ for $s\in [0, t]$ and by 
$Y_n(s):= Z_n(s-t)$ for $s\in [t, T_n+t]$. We obtain:
$$
A_{\widetilde{H}}(\uY) = A_{\widetilde{H}}(\uZ) + 
\int_0^t \left[ \widetilde{\lambda}_{\Phi^{\widetilde{H}}_s(X_0)}\big(X_{\widetilde{H}}(\Phi^{\widetilde{H}}_s(X_0))\big) - \widetilde{H}(\Phi^{\widetilde{H}}_s(X_0))\right]\,ds\,. 
$$
This implies that
$$
B_{\widetilde{H}}(X_0,X_1) \leq B_{\widetilde{H}}(\Phi^{\widetilde{H}}_t(X_0),X_1) 
+ \int_0^t \left[ \widetilde{\lambda}_{\Phi^{\widetilde{H}}_s(X_0)} \big(X_{\widetilde{H}}(\Phi^{\widetilde{H}}_s(X_0))\big) - \widetilde{H}(\Phi^{\widetilde{H}}_s(X_0))  \right]\,ds
$$
and hence we conclude that equality holds.\\
\end{Proof}

\begin{Prop}\label{symplectomorphism}
Let $\Psi: \rT^*M \longrightarrow \rT^*M$ be a symplectomorphim such that $H\circ \Psi$ is still of Tonelli type and let $\tPsi: \rT^*\tM \longrightarrow \rT^*\tM$ be its lift to the universal cover. $\tPsi$ is exact and let us denote by $S: \rT^*\tM \longrightarrow \R$ a primitive for $\tPsi^*\tilde{\lambda}-\tilde{\lambda}$, where $\tilde{\lambda}$ is the lift of the Liouville form of  $\rT^*M$ to $\rT^*\tM$. Then,
$$
B_{\widetilde{H}\circ\widetilde{\Psi}}(X_0,X_1) = B_{\tH}(\tPsi(X_0),\tPsi(X_1)) + S(X_0) - S(X_1)\,.
$$
\end{Prop}

\begin{Proof}
First of all, observe that  $\uY=(Y_n)$ is a pre-orbit of $\tH\circ \tPsi$ between $X_0$ and $X_1$ if and only if $\uZ: = \tPsi(\uY)=(\tPsi(Y_n))$ is a pre-orbit of $\tH$ between $\tPsi(X_0)$ and $\tPsi(X_1)$. As a consequence, it is enough to prove that
\beqa{formula1a}
A_{\tH\circ \tPsi}(\uY) = A_{\tH}(\uZ) + S(X_0)- S(X_1).
\eeqa
In fact, using the same notation as in Definition \ref{preorbit}:
\beqano
A_{\tH}(Z_n) &=& \sum_{n=0}^{N_n} \int_{T_n^i}^{T_n^{i+1}} \left[ \tilde{\lambda}_{Z_{n}(t)} (\dot{Z}_n(t)) - \tH(Z_n(t))\right]\,dt  \\
&=& \sum_{n=0}^{N_n} \int_{T_n^i}^{T_n^{i+1}}\left[ (\tPsi^*\tilde{\lambda})_{Y_{n}(t)}(\dot{Y}_n(t)) - (\tH\circ\tPsi)(Y_n(t))\right] \,dt  \\
&=& \sum_{n=0}^{N_n} \left( \int_{T_n^i}^{T_n^{i+1}} \left[\tilde{\lambda}_{Y_{n}(t)}(\dot{Y}_n(t)) - (\tH\circ\tPsi)(Y_n(t))\right]\,dt \;  \right.\\
&& \quad +\; \left. \int_{T_n^i}^{T_n^{i+1}} (\tPsi^*\tilde{\lambda} - \tilde{\lambda})_{Y_{n}(t)}\dot{Y}_n(t)\,dt  \right) \\
&=& A_{\tH\circ \tPsi} (Y_n) + \sum_{n=0}^{N_n} \int_{T_n^i}^{T_n^{i+1}} dS_{Y_{n}(t)}(\dot{Y}_n(t))\,dt  \\
&=& A_{\tH\circ \tPsi} (Y_n) + \sum_{n=0}^{N_n} S(Y_n(T_n^{i+1}-)) - S(Y_n(T^i_n))   \\
&=& A_{\tH\circ \tPsi} (Y_n) +  S(Y_n(T_n)) - S(Y_n(0)) +  \sum_{n=1}^{N_n} S(Y_n(T_n^{i}-) - S(Y_n(T^i_n)) \,.
\eeqano
It is now sufficient to observe that the last term in the sum goes to zero as $n$ goes to infinity. In fact, using conditions (iv) and (v)  in Definition \ref{preorbit}, we can conclude that there exists $C>0$ such that:
\beqano
\left| \sum_{n=1}^{N_n} S(Y_n(T_n^{i}-) - S(Y_n(T^i_n))\right|  \leq C \Delta(Y_n) \stackrel{n\rightarrow +\infty}{\longrightarrow} 0\,,
\eeqano
and this, together with the fact that $Y_n(T_n)\rightarrow X_1$ and $Y_n(0)\rightarrow X_0 $, allows us to  conclude (\ref{formula1a}) and hence the proof of the proposition.
\end{Proof}

Let us see how we can recover the original definition of Peierls barrier (see (\ref{Peierlsbarrier})) from this  barrier in the phase space. Recall the definitions of $c_u(H)$ introduced in (\ref{criticalvalues}) and let us define
$$
h^k(\tx_1,\tx_2) := \liminf_{T\rightarrow +\infty} \left(h_T(\tx_1,\tx_2) + kT\right)\,.
$$
It follows from (\ref{Peierlsbarrier}) that $h=h^{c_u(H)}$. Then,

\begin{Prop} \label{coincidencedefinition} For any $q,q'\in \tM$,
\beqa{coincidencebarriers}
h^k(q,q') = \min_{{\small \begin{array}{l}P\in \rT_q^*\tM\\ P'\in \rT^*_{q'}\tM \end{array}}}  B_{\tH-k}(P,P').
\eeqa
In addition, when $k=c_u(H)$, if the minimum is reached at $(P,P')$, then $P$ is a superdifferential of the function $h(\cdot, q')$ at a point $q$ and $-P'$ is a superdifferential of the function $h(q, \cdot)$ at a point $q'$. 
\end{Prop}

\begin{Proof}
Let us first prove that if $q,q'\in \tM$, then $B_{\tH-k}(P,P')\geq h^k(q,q')$ for all $P\in \rT^*_q\tM$ and 
$P'\in \rT^*_{q'}\tM$.
If $B_{\tH-k}(P,P')=+\infty$, then there is nothing to prove. Otherwise if $B_{\tH-k}(P,P')\in\R$ (resp.  
$B_{\tH-k}(P,P')=-\infty$),
for any $\e>0$ there exists a pre-orbit
$\uY=(Y_n)_n$, $Y_n:[0,T_n] \longrightarrow \rT^*\tM$ such that 
$A_{\tH-k}(\uY) \leq B_{\tH-k}(P,P') + \e$  (resp. $A_{\tH-k}(\uY) \leq -\frac{1}{\e}$).

Let us consider $q_n(s):=\pi(Y_n(s))$, where $\pi: \rT^*\tM \longrightarrow \tM$ denotes the canonical projection, and let $h_T(q,q') := \min_{\g\in C_T(q,q')} A_{\tilde{L}}(\g) $ be the finite time potential as defined in (\ref{finitetimepotential}), \ie the minimal Lagrangian action of curves in $\tM$ that connect $q$ to $q'$ in time $T>0$. These functions are equi-Lipschitz on compact regions of $\tM$ \cite[Proposition 3-4.1]{CI99}. Then, 
we have:
\beqano
A_{\tH-k}(Y_n) &=&  \sum_{n=0}^{N_n} \int_{T_n^i}^{T_n^{i+1}} \left( L(q_n(s), \dot{q}_n(s)) + k \right)\,ds \\
&\geq& \sum_{n=0}^{N_n}  h_{T_n^{i+1}-T_n^i} (q_n(T_n^i), q_n(T_n^{i+1}-)) + k (T_n^{i+1}-T_n^i).
\eeqano
Let $\s_n^i$ be a unit-speed shortest geodesic connecting $q_n(T_n^{i}-)$ and $q_n(T_n^{i})$ and let 
$\d_n^i:=\dist(q_n(T_n^{i}-),q_n(T_n^{i}))$. Adding and subtracting the action of these geodesics, we obtain:
\beqano
&& A_{\tH-k}(Y_n) \; \geq\;    \sum_{n=0}^{N_n}  \Big( h_{T_n^{i+1}-T_n^i} (q_n(T_n^i), q_n(T_n^{i+1}-))  + k (T_n^{i+1}-T_n^i)  \\
&& \qquad\qquad\qquad \qquad \pm \; \int_0^{\d_n} (\tilde{L} + k)(\s_n^{i+1},\dot{\s}^{i+1}_n)\,ds \Big) \\
&& \qquad \qquad \qquad\geq \; h_{T_n+\sum_{n=1}^{N_n}\d_n^i}(q_n(0),q_n(T_n))  + k \Big(T_n+\sum_{n=1}^{N_n}\d_n^i\Big)  \\
&& \qquad \qquad\qquad\qquad+ \; \sum_{n=1}^{N_n}  \int_0^{\d_n} (\tilde{L} + k)(\s_n^{i+1},\dot{\s}^{i+1}_n)\,ds.
\eeqano
Therefore, 
\beqano
A_{\tH-k}(\uY) &=& \liminf _{n\rightarrow +\infty} A_{\tH-k}(Y_n)\geq \; \ldots \; \geq\\
&\geq&  h^k(q,q') +  \liminf _{n\rightarrow +\infty} \sum_{n=1}^{N_n}  \int_0^{\d_n} (\tilde{L} + k)(\s_n^{i+1},\dot{\s}^{i+1}_n)\,ds\,.
\eeqano
Observe now that 
$
\lim_{n\rightarrow +\infty} \big| \sum_{n=1}^{N_n}  \int_0^{\d_n} (\tilde{L} + k)(\s_n^{i+1},\dot{\s}^{i+1}_n)\,ds\big| = 0$. As usual, this follows from property (iv) in Definition \ref{preorbit}. This concludes the proof of this inequality.
 Conversely, suppose that $T_n\rightarrow +\infty$ and $h^k(q,q')=\lim_{n\rightarrow +\infty} \left( h_{T_n}(q,q') + kT_n\right)$. Let $q_n:[0,T_n]\longrightarrow \tM$ be a Tonelli minimizer and therefore one can consider the associated  orbit of the Hamiltonian flow of $\tH$, $Y_n(s):= (q_n(s), p_n(s))$. Since the actions of these orbits are bounded, then there exists a compact subset of $\rT^*\tM$ containing the images of these curves. Up to extracting a subsequence, we can assume that:
$$
Y_n(0) \longrightarrow P \in \rT^*_q\tM \qquad {\rm and} \qquad 
Y_n(T_n) \longrightarrow P' \in \rT^*_{q'}\tM. 
$$
Hence, the sequence $\uY=(Y_n)$ is a pre-orbit between $P$ and $P'$. Moreover,
$$
A_{\tH-k}(\uY) = \lim_{n\rightarrow +\infty} A_{\tH-k}(Y_n) = \lim_{n\rightarrow +\infty} \int_0^{T_n} \Big(L(q_n,\dot{q}_n)+k\Big)\,ds = h^k(q,q')
$$
and therefore
$B_{\tH-k}(P,P') \leq h^k(q,q')$ and this completes the proof of (\ref{coincidencebarriers}).

Suppose now that $P\in \rT_q^*\widetilde{M}$ and $P'\in \rT_{q'}^*\widetilde{M}$ are two points such that
$h_{\widetilde{H}}(q,q') =   B_{\tH-c_u(H)}(P,P')$ and denote by $q(s)$ the projection to $\widetilde{M}$ of the orbit $\Phi^{\widetilde{H}}_s(P)$.
Using Proposition \ref{Prop2.5Bernard} and the properties of $h_{\widetilde{H}}$, we obtain:
\beqano
B_{\tH-c_u(H)}(P,P') &=& \th_{\tH-c_u(H)}(\Phi^{\widetilde{H}}_s(P),P') \; \\
&&  +\; 
 \int_0^s \left(
\widetilde{\lambda}_{\Phi^{\widetilde{H}}_s(P)}\left( X_{\widetilde{H}}(\Phi^{\widetilde{H}}_s(P)) \right) - \tH(\Phi^{\widetilde{H}}_s(P) + c_u(H))
\right)\, dt\;  \\
&\geq& h_{\tH}(q(s), q') + \int_0^s \left( \widetilde{L}(q(t), \dot{q}(t)) + c_u(H) \right)\,dt \; \\
&\geq& h_{\tH}(q,q') = B_{\tH-c_u(H)}(P,P')\,.
\eeqano

Therefore, all the above inequalities are equalities and consequently our curve is an action-minimizing curve:
$$
h_{\tH}(q,q') = \min_s \left( h_{\tH}(q(s), q') + \int_0^s \left( \widetilde{L}(q(t), \dot{q}(t)) + c_u(H) \right)\,dt \right)\,.
$$
It follows from Fathi's work \cite{Fathibook} that $-P$ is then a superdifferential of the function
$h_{\tH}(\cdot,q')$ at $q$. Similarly for the other property.

\end{Proof}

Let us denote now $\tilde{m}(H)= \inf_{X\in \rT^*M} B_{\tH}(X,X)$. It is easy to check that $\tilde{m}(H) \in \{-\infty\}\cup [0,+\infty]$.

\begin{Prop}\label{characc_u}
$c_u(H)= \sup\{k\in\R:\; \tilde{m}(H-k) > -\infty\} = \inf\{k\in\R:\; \tilde{m}(H-k) \geq 0\}$.
\end{Prop}

The proof simply follows from the definition of $c_u(H)$ (see subsection \ref{sec1.2} and (\ref{criticalvalues})) and Proposition \ref{coincidencedefinition}.\\

Proposition \ref{thm1}, \ie the symplectic invariance of $c_u(H)$, can now be proved in a different way using this new characterization of $c_u(H)$ and  Proposition \ref{symplectomorphism}.\\

Let us prove now the following result.\\

\begin{Prop}
$\cA^*_{\widetilde{H}} = \{P\in \rT^*\widetilde{M}:\; B_{\tH-c_u(H)}(P,P)=0\}$.\\
\end{Prop}

\begin{Rem}
Observe that now Theorem \ref{invarianceAubry} will follow from this proposition and Proposition \ref{symplectomorphism}.\\
\end{Rem}

\begin{Proof}
[$\supseteq$] Let $P\in \rT^*\widetilde{M}$ such that $B_{\tH-c_u(H)}(P,P)=0$ and let us denote by $q=\pi (P)$. Hence, $h_{\widetilde{H}}(q,q)\leq 0$ and therefore $h_{\widetilde{H}}(q,q)= 0$. This implies that
$q\in \cA_{\widetilde{H}}$ and consequently $h_{\widetilde{H}}(q,\cdot)$ is differentiable at $q$ and 
$(q, \partial_2h_{\widetilde{H}}(q,q)) \in \cA^*_{\widetilde{H}}$.
Now, recall from Proposition \ref{coincidencedefinition} that  since $B_{\tH-c_u(H)}(P,P)=h_{\widetilde{H}}(q,q)  $, then $P$ is a superdifferential of $h_{\tH}(q,\cdot)$ at $q$ and hence $P=\partial_2h_{\widetilde{H}}(q,q)$.

[$\subseteq$] Let $P\in \cA^*_{\tH}$, then $h_{\tH}(q,q)=0$, where $q=\pi(P)$, and $h_{\tH}(q,\cdot)$ and $h(\cdot, q)$ are both differentiable at $q$. Therefore, $P=\partial_2h_{\tH}(q,q)= -\partial_1h_{\tH}(q,q)$. In particular we know from Proposition \ref{coincidencedefinition} that if $X, X'\in \rT^*\widetilde{M}$ are such that $B_{\tH-c_u(H)} = h_{\tH}(q,q)$, then $-X$ is a superdifferential of  $h_{\tH}(q,\cdot)$ at $q$ and 
$X'$ is superdifferential of $h_{\tH}(\cdot,q)$ at $q$. Hence, $X=X'=\partial_2h_{\tH}(q,q)$, which concludes the proof.\\
\end{Proof}

\noindent {\bf Question IV}: 
Similar questions might be also asked for the Ma\~n\'e set associated to $\widetilde{H}$: is it true that it is symplectic invariant? The main problem in proving this is represented by the fact that, differently from what happens in the compact case,  in our setting the Aubry set might be empty. However, if the Aubry set is non-empty, then the proof would follow essentially what already done in Bernard's article \cite[Section 2.10 and Corollary 3.7]{B07}.

\section{Some examples} \label{examples}

In this section we determine the Aubry set of the universal cover in some examples.
We also exhibit invariant measures with zero homotopy in cases where $c_u<c_a$.
Before we describe the examples it is convenient to state and prove a simple lemma that will allow
us to compute the Peierls barrier.

Let $M$ be a closed manifold with a Tonelli Lagrangian $L$ and fix $T>0$. By Tonelli's theorem, there exists a closed contractible orbit $\tau:[0,T]\to M$ which minimizes the action $A_{L}$ over the free loop space of closed contractible curves defined on $[0,T]$.  As before let $\Pi_{u}:\widetilde{M}\to M$ be the universal cover.

\begin{Lem} Let $x\in\widetilde{M}$ be such that $\Pi_{u}(x)\in \tau([0,T])$. Then
\[h_{\widetilde{H}}^T(x,x)=A_{L}(\tau).\]
\label{lemma:aux}
\end{Lem}

\begin{proof} If $\gamma:[0,T]\to\widetilde{M}$ is an absolutely continuous loop based at $x$, then
obviously $\Pi_{u}\circ\gamma$ is a closed contractible curve in $M$ and
\[A_{L}(\gamma)\geq A_{L}(\tau).\]
Since $\tau$ lifts to a closed loop based at $x$, then the lemma follows immediately.

\end{proof}

All the Lagrangians considered here have the form 
\[L(x,v)=\frac{1}{2}|v|_{x}^2+\theta_{x}(v)\]
for some Riemannian metric $|\cdot|_{x}$ and a smooth 1-form $\theta$. The corresponding Hamiltonian
is $H(x,p)=\frac{1}{2}|p-\theta_{x}|^2_{x}$.  These examples have already been considered in
\cite{CFP} but their Aubry sets and Peierls barriers were not computed there.

\subsection{Example with $c_u=c_a$ but $\cA_{\widetilde{H}} =\emptyset$}
Let $G$ be the 3-dimensional Heisenberg group of matrices
\[\left(\begin{array}{ccc}

1&x&z\\
0&1&y\\
0&0&1\\

\end{array}\right),\]
where $x,y,z\in \R$. If we identify $G$ with $\R^3$, then the product is
\[(x,y,z)\star(x',y',z')=(x+x',y+y',z+z'+xy').\]
We let $\Gamma$ be the lattice of those matrices with
$x,y,z\in\Z$. Then $M=\Gamma\setminus G$ is a closed 3-dimensional nilmanifold.
We consider the Lagrangian
\[L=\frac{1}{2}(\dot{x}^2+\dot{y}^2+(\dot{z}-x\dot{y})^2)+\dot{z}-x\dot{y}.\]
It is easy to check that $L$ is invariant under the left action of $G$, hence it descends to $M$.
Various properties of this systems were proved in \cite{CFP}. Here we need:
\begin{enumerate}
\item $c_u=c_a=1/2$ (\cite[Lemma 6.8]{CFP});
\item there is a closed contractible orbit with energy $k>0$ if and only if $k<1/2$.
Moreover the (prime) closed contractible orbits with energy $k$ have $A_{L+k}$-action
equal to $2\pi(1-\sqrt{1-2k})$ and period $T=2\pi/(\sqrt{1-2k})$, see \cite[Lemma 6.7]{CFP} (with the notation of \cite{CFP},  $\Omega(v)$ corresponds precisely to the $A_{L+k}$-action, as it is easy to check).
\end{enumerate}

We now show:

\begin{Lem} For any $x\in\R^3$, $h_{\widetilde{H}}(x,x)=2\pi$.
\end{Lem}

\begin{proof} Since $G$ acts transitively,  the function $x\mapsto h_{\widetilde{H}}(x,x)$ is constant.

Let $\tau_{T}$ be one of the prime closed orbits described in item (2) above. Then
\[A_{L+1/2}(\tau_{T})=A_{L+k}(\tau_{T})+(1/2-k)T= {2\pi \left(1-\frac{\pi}{T} \right)}.\]
Using item (2) above, we can list all closed contractible orbits with period $T$:
\begin{itemize}
\item constant curves defined on $[0,T]$;
\item $\tau_T$;
\item iterates $n\tau_{T/n}$ where $n$ is a positive integer such $n\leq T/2\pi$ (the reason for this latter condition comes from the fact that, as remarked before, only energy levels with $k<1/2$ contain such orbits and their periods are determined by the energy itself).
\end{itemize}

The constant curves have $A_{L+1/2}$-action
equal to $T/2$ and the iterates have action $A_{L+1/2}(n\tau_{T/n})={2\pi n \left(1-\frac{n\pi}{T}\right)}$.
Hence for $T$ large the $\tau_{T}$ are the Tonelli minimizers of the action on the free
loop space; {in fact, if $T$ is large,
$2\pi n \left(1-\frac{n\pi}{T}\right)>2\pi\left(1-\frac{\pi}{T}\right)$ for $1<n\leq T/2\pi$.
}
By Lemma \ref{lemma:aux} we conclude that for all $T$ large
\[h_{\widetilde{H}}^T(x,x)+T/2={2\pi \left(1-\frac{\pi}{T} \right)}\]
and the lemma follows by letting $T$ go to infinity.
\end{proof}

Besides showing that $h_{\widetilde{H}}<\infty$ this is also shows that $\cA_{\widetilde{H}} =\emptyset$ as claimed. On the other hand, on the abelian cover $\overline{M}$, we have:

\begin{Lem}
 $\cA_{\overline{H}}=\overline{M}$.
\end{Lem}

\begin{proof} Let $Z\subset\Gamma$ be the center of $\Gamma$. It consists of all elements of the form
$(0,0,n)$ for $n\in\Z$. Then the abelian cover $\overline{M}=Z\setminus G$.  Note that
\[L+1/2=\frac{1}{2}(\dot{x}^2+\dot{y}^2+(\dot{z}-x\dot{y}+1)^2)\]
hence the curves $t\mapsto (x,y,z-t)$ are solutions to the Euler-Lagrange equations and have energy $1/2$. They project to closed curves in $\overline{M}$ with period $1$ and have zero $A_{\overline{L}+1/2}$-action.
It follows that $h_{\overline{H}}(p,p)=0$ for all $p\in \overline{M}$.

\end{proof}

\subsection{Example with $c_u<c_a$ and $\cA_{\widetilde{H}} =\emptyset$}
In \cite{PP97} the authors provided an example of a Tonelli Lagrangian on a closed orientable surface of genus two for which $c_u<c_a$. It is possible to construct many other examples of this kind also in higher dimension, as it was shown in \cite{CFP}.
Here we discuss a homogeneous example considered in \cite[Section 6.3]{CFP}.

We identify $PSL(2,\R)$ with $S\H$, the unit sphere bundle of the hyperbolic plane $\H:=\R\times (0,+\infty)$ with the usual Poincar\'e metric of curvature $-1$, given by:
$ds^2 = \frac{1}{y^2} (dx^2 + dy^2)$. 
We consider a cocompact lattice $\Gamma$ and we let $M:=\G \backslash PSL(2,\R)$.

We consider coordinates $(x,y,\theta)$  in $S\H$, where $(x,y)$ represents points in $\H$, while $\theta$ parametrizes the circle fibres. Moreover, we endow $S\H$ with its Sasaki metric:
$$
ds^2 = \frac{1}{y^2} (dx^2 + dy^2 + (yd\theta  + dx)^2).
$$
The 1-form $d\theta+\frac{dx}{y}$ is left-invariant, hence the following Lagrangian is also left-invariant
\[L=\frac{1}{2y^2}(\dot{x}^2+\dot{y}^2+(y\dot{\theta}+\dot{x})^2)+\dot{\theta}+\frac{\dot{x}}{y}\]
and therefore it descends to $M$.

Various properties of this systems were proved in \cite{CFP}. Here we need:
\begin{enumerate}
\item $c_u=1/4$ and $c_a=1/2$ (\cite[Lemma 6.11]{CFP});
\item there is a closed contractible orbit with energy $k>0$ if and only if $k<1/4$.
Moreover the (prime) closed contractible orbits with energy $k$ have $A_{L+k}$-action
equal to $\pi(1-\sqrt{1-4k})$ and period $T=2\pi/(\sqrt{1-4k})$, see \cite[Lemma 6.14]{CFP}.
\end{enumerate}

Observe that $PSL(2,\R)=S\H$ is not simply connected; this will cause no problem though.

We now show:

\begin{Lem} For any $x\in \widetilde{S\H}$, $h_{\widetilde{H}}(x,x)=\pi$.
\end{Lem}

\begin{proof} Since $\widetilde{PSL}(2,\R)$ acts transitively,  the function 
$x\mapsto h_{\widetilde{H}}(x,x)$ is constant.

Let $\tau_{T}$ be one of the prime closed orbits described in item (2) above. Then
\[A_{L+1/4}(\tau_{T})=A_{L+k}(\tau_{T})+(1/4-k)T=\pi-\frac{\pi\sqrt{1-4k}}{2}=\pi-\frac{\pi^2}{T}.\]
Using item (2) above we can list all closed contractible orbits with period $T$:
\begin{itemize}
\item constant curves defined on $[0,T]$;
\item $\tau_T$;
\item iterates $n\tau_{T/n}$ where $n$ is a positive integer such $n\leq T/2\pi$
(the reason for this latter condition comes from the fact that, as remarked before, only energy levels with $k<1/4$ contain such orbits and their periods are determined by the energy itself).
\end{itemize}

The constant curves have $A_{L+1/4}$-action
equal to $T/4$ and the iterates have action $A_{L+1/4}(n\tau_{T/n})=n\pi(1-\frac{n\pi}{T})$.
But, for $T$ large, $n\pi(1-\frac{n\pi}{T})>\pi(1-\frac{\pi}{T})$ for $1<n\leq T/2\pi$.
Hence for $T$ large the $\tau_{T}$ are the Tonelli minimizers of the action on the free
loop space and by Lemma \ref{lemma:aux} we conclude that for all $T$ large
\[h_{\widetilde{H}}^T(x,x)+T/4=\pi(1-\frac{\pi}{T})\]
and the lemma follows by letting $T$ go to $\infty$.
\end{proof}

As in the previous example, besides showing that $h_{\widetilde{H}}<\infty$ this is also shows that $\cA_{\widetilde{H}} =\emptyset$ as claimed.\\

\vspace{10 pt}

\noindent{\bf Minimizing measures with zero homotopy.} We now describe, in this specific example, all ergodic minimizing invariant measures with zero homotopy. Let $\mu$ be such a measure.
Since $c_u=1/4$ and $\mu$ is ergodic its support must be contained in the energy level $E=1/4$
(cf. Proposition \ref{energymeasure}). 
Recall that the corresponding Hamiltonian vector field is given by (see \cite[Section 6.3]{CFP}):
$$
X_H = \left\{\begin{array}{lll}
\dot{x}= y(y p_x - p_\theta), && \dot{p}_x = p_y,\\
\dot{y}= y^2 p_y, && \dot{p}_y = (-yp_x+p_\theta)(p_x+1/y)-yp_y^2,\\
\dot{\theta}= 2p_\theta-yp_x, && \dot{p}_\theta = 0.\\
\end{array}
\right.
$$

Then, the function $f=\dot{\theta}+\dot{x}/y$ is clearly a first integral of the system hence it must be a constant $a$ on the support of $\mu$. Using that $\mu$ is minimizing and the explicit form of $L=E+f$ we deduce
\[A_{L}(\mu)=-1/4=1/4+a\]
and thus $a=-1/2$. To describe the flow for $k=1/4$ and $a=-1/2$ it is easier to pass to the Hamiltonian
setting and introduce left-invariant coordinates $(x,y,\theta,\pa,\pb,\pc)$  in $T^*PSL(2,\R)$ as in \cite[Section 6.3]{CFP}.  If we let
\begin{align*}
\pa&=(yp_{x}-p_{\theta})\cos\theta+yp_{y}\sin\theta,\\
\pb&=-(yp_{x}-p_{\theta})\sin\theta+yp_{y}\cos\theta,\\
\pc&=p_{\theta}
\end{align*}
then
\[H=\frac{1}{2}(\pa^2+\pb^2+(\pc-1)^2).\]

In terms of these left-invariant coordinates, the Hamiltonian vector field becomes (see \cite[Section 6.3]{CFP}):
$$
X_H = \left\{\begin{array}{lll}
\dot{x}= y(p_\a \cos\theta - p_\beta \sin\theta ),         && \dot{p}_\alpha = 2p_\beta p_\gamma + p_\beta,\\
\dot{y}= y(p_\a \sin\theta + p_\beta \cos\theta ),     && \dot{p}_\beta = -2p_\alpha p_\gamma - p_\alpha,\\
\dot{\theta}= p_\gamma-p_\alpha\cos\theta + p_\beta \sin \theta, && \dot{p}_\gamma = 0.\\
\end{array}
\right.\\
$$

Using the above expressions, a simple calculation now shows that $-a=\pc=1/2$ and that $\pa$ and $\pb$ must be constant if $k=1/4$ and $\pc=1/2$. Hence the orbits of $H$ for $k=1/4$ and $\pc=1/2$ are orbits of the right action of 1-parameter subgroups of $PSL(2,\R)$ determined by $(\pa,\pb,1/2)$ such that $\pa^2+\pb^2=1/4$.  It is straighforward to check that these 1-parameter subgroups are all parabolic ({\it i.e.} horocycle flows). These flows are known to be uniquely ergodic (as proved by H. Furstenberg in  \cite{Furstenberg}),  and the unique invariant probability measure is the normalised Lebesgue measure $\mu_{\pa,\pb}$ on $\Gamma\setminus PSL(2,\R)$. It is easy to check that these measure have zero homotopy: they are weak limits of the probability measures supported on the closed
orbits $\tau_T$ as $k\to 1/4$ (or equivalently as $T\to\infty$). Hence our measure $\mu=\mu_{\pa,\pb}$ for some $(\pa,\pb)$.  Observe that we get a whole circle worth of minimizing measures
with zero homotopy and the union of their supports is {\it not} a graph (the support of each ergodic component is a graph though). This, quite surprisingly, is in contrast with Mather's celebrated graph theorem \cite{Mat91}.

\color{black}
%%%%%%%%%%%%%%%%%%%REFERENCES %%%%%%%%%%%%%%%%%%%%%%%%%%%%%%%%%%%%%%%%%%%%%%
\def\cprime{$'$}

\vspace{1.truecm}

\end{document}